\documentclass[psamsfonts]{amsart}

\usepackage{enumerate}
\usepackage{url}
\usepackage[font=small,labelfont=bf]{caption}
\usepackage[all,arc]{xy}
\usepackage{tabularx}
\usepackage{adjustbox}

\usepackage{amssymb,latexsym}
\usepackage{amsmath,amsthm}
\usepackage{amsfonts,mathrsfs}
\usepackage{mathtools}

\usepackage{tikz}
\usetikzlibrary{shapes.geometric, arrows, positioning,decorations.pathreplacing,calligraphy}

\tikzstyle{bloc} = [rectangle, rounded corners, 
minimum width=3cm, 
minimum height=1cm,
text centered, 
draw=black, 
fill=airforceblue!30]

\tikzstyle{decision} = [diamond,
minimum width=3cm, 
minimum height=1cm, 
text centered, 
draw=black, 
fill=airforceblue!30]
\tikzstyle{arrow} = [thick,->,>=stealth]

\tikzstyle{io} = [trapezium, 
trapezium stretches=true, % A later addition
trapezium left angle=70, 
trapezium right angle=110, 
minimum width=3cm, 
minimum height=1cm, text centered, 
draw=black, fill=blue!30]

\tikzstyle{process} = [rectangle, 
minimum width=3cm, 
minimum height=1cm, 
text centered, 
text width=3cm, 
draw=black, 
fill=orange!30]

\usepackage{tikz-cd}
\usepackage{todonotes}

\usepackage[all]{xy}
\usepackage{graphicx}

\usepackage{hyperref}
\hypersetup{
    colorlinks,
    citecolor=blue,
    filecolor=blue,
    linkcolor=blue,
    urlcolor=blue
}

\usepackage{enumitem}

\usepackage{listings}
\usepackage{color}

\definecolor{dkgreen}{rgb}{0,0.6,0}
\definecolor{gray}{rgb}{0.5,0.5,0.5}
\definecolor{mauve}{rgb}{0.58,0,0.82}

\newtheorem{thm}{Theorem}[section]
\newtheorem{cor}[thm]{Corollary}
\newtheorem{prop}[thm]{Proposition}
\newtheorem{obs}[thm]{Observation}
\newtheorem{lem}[thm]{Lemma}

\newtheorem{quest}[thm]{Question}
\newtheorem{theorem}[thm]{Theorem}

\newtheorem{corollary}[thm]{Corollary}
\newtheorem{lemma}[thm]{Lemma}

\newtheorem*{theorem*}{Theorem}

\theoremstyle{definition}
\newtheorem{defn}[thm]{Definition}

\newtheorem{fact}[thm]{Fact}

\theoremstyle{definition}
\newtheorem{definition}[thm]{Definition}

\theoremstyle{remark}
\newtheorem{rem}[thm]{Remark}

\newtheorem{remark}[thm]{Remark}

\newcommand{\brac}[2]{\left[ {#1} , {#2} \right]}

\makeatletter
\let\c@equation\c@thm
\makeatother
\numberwithin{equation}{section}

\newcommand\N{\mathbb{N}}

\newcommand\F{\mathbb{F}}

\newcommand\M{\mathbb{M}}

\newcommand\LL{\mathscr{L}}

\newcommand{\ACF}{\mathrm{ACF}}

\newcommand{\tp}{\mathrm{tp}}

\def\seq{\subseteq}

\newcommand{\set}[1]{\left\{ {#1} \right\}}
\newcommand{\vect}[1]{\langle {#1} \rangle}

\newcommand{\ol}[1]{\overline{#1}}

\definecolor{airforceblue}{rgb}{0.36, 0.54, 0.66}

\def\Ind{\setbox0=\hbox{$x$}\kern\wd0\hbox to 0pt{\hss$\mid$\hss}
\lower.9\ht0\hbox to 0pt{\hss$\smile$\hss}\kern\wd0}
\def\Notind{\setbox0=\hbox{$x$}\kern\wd0\hbox to 0pt{\mathchardef
\nn=12854\hss$\nn$\kern1.4\wd0\hss}\hbox to
0pt{\hss$\mid$\hss}\lower.9\ht0 \hbox to 0pt{\hss$\smile$\hss}\kern\wd0}
\def\ind{\mathop{\mathpalette\Ind{}}}

\def\indi#1{\mathop{\ \ \hbox to 0ex{\hss$\vert^{\hbox to 0ex{$\scriptstyle#1$\hss}}$\hss}
\lower1ex\hbox to 0ex{\hss$\smile$\hss}\ \ }}

\def\nindi#1{\mathop{\ \ \hbox to 0ex{\hss$\!\not{\vert}^{\hbox to 0ex{$\scriptstyle\,#1$\hss}}$\hss}
\lower1ex\hbox to 0ex{\hss$\smile$\hss}\ \ }}

\renewcommand{\models}{\vDash}

\title{A two-sorted theory of nilpotent Lie algebras}
\author[C. d'Elb\'{e}e]{Christian d\textquoteright Elb\'ee$^\dagger$}
\address{University of the Basque Country, Department of Mathematics (Leioa) / Institute for Logic, Cognition, Language and Information (Donostia-San Sebastián), Spain}
\email{christian.delbee@ehu.eus}
\urladdr{\href{http://choum.net/\textasciitilde chris/page\textunderscore perso/}{http://choum.net/\textasciitilde chris/page\textunderscore perso/}}

\author[I. M\"uller]{Isabel M\"uller$^\ddagger$}
\address{Department of Mathematics and Actuarial Science \\
The American University in Cairo \\ Egypt }
\email{isabel.muller@aucegypt.edu}
\urladdr{\href{https://www.aucegypt.edu/fac/isabel}{https://www.aucegypt.edu/fac/isabel}}

\author[N. Ramsey]{Nicholas Ramsey$^\S$}
\address{Department of Mathematics \\
University of Notre Dame\\
 USA}
\email{sramsey5@nd.edu}
\urladdr{\href{https://math.nd.edu/people/faculty/nicholas-ramsey/}{https://math.nd.edu/people/faculty/nicholas-ramsey/}}

\author[D. Siniora]{Daoud Siniora}
\address{Department of Mathematics and Actuarial Science \\The American University in Cairo\\
 Egypt}
\email{daoud.siniora@aucegypt.edu}
\urladdr{\href{https://sites.google.com/view/daoudsiniora/}{https://sites.google.com/view/daoudsiniora/}}

\date{\today}

\begin{document}

\maketitle

\begin{abstract}
We prove the existence of a model companion of the two-sorted theory of $c$-nilpotent Lie algebras over a field satisfying a given theory of fields. We describe a language in which it admits relative quantifier elimination up to the field sort. Using a new criterion which does not rely on a stationary independence relation, we prove that if the field is NSOP$_1$, then the model companion is NSOP$_4$. We also prove that if the field is algebraically closed, then the model companion is $c$-NIP.
\end{abstract}

\setcounter{tocdepth}{1}
\tableofcontents

\section{Introduction}

In the early 70s, the model theory of $c$-nilpotent Lie algebras over an infinite field $K$ was studied by Macintyre and Saracino, who showed that there is no model companion  \cite{macintyre1974existentially, macintyre1974existentially2}.  Here, we consider the model theory of $c$-nilpotent Lie algebras over infinite fields and prove that there is a model companion.  The difference consists in the language we use to encode these structures.  Macintyre and Saracino viewed Lie algebras over $K$ as structures in the usual one-sorted language of $K$-vector spaces enriched with a symbol for the Lie bracket, in which scalar multiplication is given by a unary function for each scalar in $K$.  We circumvent their non-existence results by changing the language, introducing sorts $K$ and $V$ for the field and vector space, respectively, and encoding scalar multiplication as a binary function $K \times V \to V$.  This was inspired by Granger's approach to vector spaces equipped with a bilinear form over an algebraically closed field \cite{granger1999stability}. We axiomatize the model companion of the two-sorted theory of $c$-nilpotent Lie algebras over a field and we exhibit a language in which this theory has quantifier elimination. 

The primary tool that we use for proving elimination of quantifiers is the main theorem of \cite{d2023model}, which proved that the class of $c$-nilpotent Lie algebras over a field $K$ satisfies a natural version of free amalgamation.  This result, based on earlier results of Baudisch \cite{Baudisch4} and of Maier \cite{Maierexpp}, was immediately specialized to the case of a finite field $K = \mathbb{F}_{p}$, in which case $c$-nilpotent Lie algebras with $c<p$ correspond to $c$-nilpotent groups of exponent $p$ via the Lazard correspondence. However, here we use this theorem in its full generality to analyze nilpotent Lie algebras in their own right.  In addition to axiomatizing a model companion, this amalgamation result enables a treatment of $c$-nilpotent Lie algebras over arbitrary theories of fields, allowing us to give quantifier elimination \emph{relative} to the field. As a key application, we show that ultraproducts of Fraïssé limits of finite $c$-nilpotent Lie algebras over $\mathbb{F}_{p}$, viewed naturally as two-sorted structures in this language, yield models of the theory of generic $c$-nilpotent Lie algebras over pseudo-finite fields.  The proof of quantifier elimination parallels the proofs of recent quantifier elimination results for certain linear structures in \cite{abdaldaim2023higher}, \cite{ChernikovHempel3}, and \cite{nickandlotte}. 

The two-sorted theory of nilpotent Lie algebras offers an especially valuable case study for the emerging structure theory for NSOP$_{4}$ theories. In \cite{d2023model}, we included the following table, suggesting that nilpotent Lie algebras might play a role for NSOP$_{4}$ analogous to that played by vector spaces with bilinear forms over finite and algebraically closed fields in simple and NSOP$_{1}$ theories, respectively.

\small{
\begin{table}[htbp]
  \centering
  %\caption{Table Title}
  \label{tab:my-table}
  \begin{adjustbox}{center}
  \begin{tabularx}{\textwidth}{|X|X|X|X|}
  \hline
    Stable & Simple & NSOP$_{1}$ & NSOP$_{4}$ \\
    \hline
    ACF & Psf/ACFA & $\omega$-free PAC fields & Curve-excluding fields \\
    \hline
     Vector spaces & $\mathbb{F}_{p}$-vector spaces with a bilinear map & Vector spaces over ACF with a bilinear map & Nilpotent Lie algebras \\
    \hline 
    Equivalence relations & Random graph & Parameterized equivalence relations & Henson graphs  \\
    \hline 
  \end{tabularx}
  \end{adjustbox}
\end{table}
}
\normalsize
\noindent Recently, there has emerged something like a `standard method' for showing that a theory is NSOP$_{4}$, which proceeds by first establishing that the theory has a stationary independence relation coming from a notion of free amalgamation.  First developed by Patel \cite{patel2006family}, this approach has been systematized by Conant \cite{conant2017axiomatic} and generalized further by Mutchnik \cite{mutchnik2022conant}.  It was applied in \cite{johnson2023curve} and in \cite{d2023model} to show that curve-excluding fields and generic $c$-nilpotent groups of exponent $p$ are NSOP$_{4}$, respectively.  This approach is, however, limited by the fact that there are theories with no stationary independence relations at all and there is considerable interest in developing tools for showing that such theories are NSOP$_{4}$.  Very recently, some such methods were developed by Miguel Gomez to analyze certain $3$-tournaments in \cite{miguel2024classification}.  Although there is a stationary independence relation for the model companion of $c$-nilpotent Lie algebras, we observe that there is no such relation in the generic theory of $c$-nilpotent Lie algebras over pseudo-finite fields. As a consequence, we use a new approach which replaces stationarity by a version of the independence theorem, paralleling the generalization of non-forking independence from stable to simple theories. As a consequence, we are able to show that the generic theory of $c$-nilpotent Lie algebras over an NSOP$_{1}$ field is NSOP$_4$.  We suspect that this approach could be useful in future classification results. 

As a final application, we show that the two-sorted model companion of $c$-nilpotent Lie algebras is $c$-dependent and $(c-1)$-independent, for all $c \geq 2$. This follows the parallel result of \cite{d2023model} for $c$-nilpotent groups of exponent $p$, which combined quantifier elimination with the Composition Lemma of Chernikov and Hempel \cite{ChernikovHempel3} to conclude $c$-dependence, though the analysis of terms in this theory is more complicated, due to the interaction between the field and vector space sorts. For the $c = 2$ case, we are, moreover, able to conclude that the theory is NFOP$_{2}$, a notion of ternary stability introduced by Terry and Wolf \cite{terry2022irregular}, later developed as a $k$-ary notion, NFOP$_k$, in \cite{abdaldaim2023higher}.

\section{A two-sorted language for nilpotent Lie algebras} \label{two sorted section}

\subsection{Lazard Lie algebras}

\begin{defn}
    A \textit{Lie algebra} $L$ over a field $\F$ is a vector space $L$ over $\F$ equipped with a binary operation $[\cdot,\cdot]:L\times L\to L$, called a \textit{Lie bracket}, satisfying the following properties for every $a,b,c\in L$ and $\mu \in \F$:
    \begin{itemize}\itemsep3pt
        \item  $[a,a] = 0$; \hfill (Alternativity)
        \item  $\brac{a+b}{c} = \brac{a}{c}+\brac{b}{c}$,  \hfill (Bilinearity) \\ 
        $\brac{a}{b+c} = \brac{a}{b}+\brac{a}{c}$,\\
        $[\mu a,b]=\mu[a,b]=[a, \mu b]$;
        \item $\brac{a}{\brac{b}{c}}+\brac{b}{\brac{c}{a}}+\brac{c}{\brac{a}{b}}=0$. \hfill (Jacobi identity) 
    \end{itemize}
\end{defn}

A subspace $U\subseteq L$ is called a \textit{Lie subalgebra} of $L$ if $U$ is closed under the Lie bracket. Given two subsets $A,B$ of $L$, we denote by $[A,B]$ the vector span of $\set{[a,b]\mid (a,b)\in A\times B}$. A subalgebra $I\subseteq L$ is called an \textit{ideal} of $L$ if $\brac{I}{L}\subseteq I$. 

We define inductively the \textit{lower central series} of $L$ as follows:
    \begin{itemize}
        \item $L_1 = L;$
        \item $L_{n+1} = [L_n,L] \textrm{ for } n\geq 1.$
    \end{itemize}
Note that each $L_n$ is an ideal of $L$. A Lie algebra $L$ is \textit{nilpotent of class $c$} if $c$ is the least integer such that
    \[L = L_1\supseteq L_2\supseteq \ldots \supseteq L_c \supseteq L_{c+1} = 0.\]
%If $L$ is nilpotent of class $c$, then $L_c \subseteq Z(L)$, the center of $L$. 

\begin{defn}
A sequence of subalgebras $(L_{i})_{1 \leq i \leq c+1}$ is a \emph{Lazard series} of a Lie algebra $L$ if 
$$
L = L_{1} \geq L_{2} \geq \ldots \geq L_{c+1} = 0
$$
and
$$
[L_{i},L_{j}] \subseteq L_{i+j}
$$
for all $i,j$, where we stipulate $L_k = 0$ for all $k > c$.  
\end{defn}

Note that if $L$ is a Lie algebra with a Lazard series $(L_{i})_{1 \leq i \leq c+1}$, then $L$ must be of nilpotence class at most $c$.  The lower central series is an example of a Lazard series. 

\begin{defn}
    We define a \emph{Lazard Lie algebra (LLA)} $(L,\overline{L})$ to be a Lie algebra $L$ with a distinguished Lazard series $\overline{L} = (L_{i})_{1 \leq i \leq c+1}$.  We will not always explicitly display the Lazard series $\overline{L}$ when referring to an LLA $(L,\overline{L})$, referring to it instead simply as $L$.  
\end{defn}

\begin{definition} \label{def:freeamalgambaudisch}
    Let $A,B,C$ be LLAs of nilpotency class at most $c$ over a fixed field $\mathbb{F}$ with embeddings $f_0 :C\to A, g_0:C\to B$. We say that the (at most) $c$-nilpotent LLA $S$ over $\mathbb{F}$ is a \textit{free amalgam of $A$ and $B$ over $C$} if there are embeddings $f_1:A\to S$, $g_1:B\to S$ with $f_1\circ f_0 = g_1\circ g_0$ such that the following three conditions hold, for $A' = f_1(A), B' = g_1(B), C' = (f_1\circ f_0)(C)$:
    \begin{enumerate}
        \item $S = \vect{A'B'}$;
        \item \textit{(Strongness)} $A'\cap B' = C'$;
        \item \textit{(Freeness)} for any at most $c$-nilpotent LLA $D$ and any LLA \textit{homomorphisms} $f:A\to D$ and $g:B\to D$, there exists a (unique) $h : S\to D$ such that the following diagrams commute. 
        
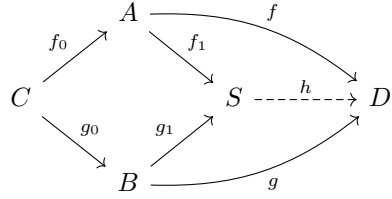
\begin{center}\begin{tikzcd}
& A \ar[dr,"f_1"] \ar[drr, "f", bend left=20]
&
&[1.5em] \\
C \ar[ur,"f_0"] \ar[dr,"g_0"]
&
& S  \ar[r, "h",dashed]
& D \\
& B \ar[ur, "g_1"]\ar[urr, "g"', bend right=20]
&
&
\end{tikzcd}\end{center}
    \end{enumerate}
    We denote the free amalgam $S$ by $A\otimes_C B$ (the use of the definite article is justified because $S$ is unique up to isomorphism, see \cite[Remarks 4.15 and 4.16]{d2023model}).
\end{definition}

\begin{definition}
    If $A,B,C$ are LLAs over a fixed field $\mathbb{F}$ and are subalgebras of a common LLA over $\mathbb{F}$, we define
    \[A\indi \otimes _C B\iff \vect{ABC}\cong \vect{AC}\otimes_C \vect{BC},\] where $\langle X \rangle$ denotes the LLA over $\mathbb{F}$ generated by $X$.  
\end{definition}

Building off of earlier work of Baudisch \cite{Baudisch4}, the following was established in \cite{d2023model}:

\begin{fact}\label{fact:basicpropertiesoffreeindependence} \cite[Proposition 4.22, Corollary 4.40]{d2023model}
    The relation $\indi \otimes$ satisfies symmetry, invariance, monotonicity, base monotonicity, stationarity, transitivity, full existence (for all $A,B,C$ there exists $A'\equiv_C A $ such that $A'\indi \otimes _C B$). In other words, $\indi \otimes$ is a stationary independence relation in the sense of \cite{tentzieglerurysohn}.
\end{fact}

\subsection{The two-sorted language}
We define a language $\LL_{K,V,c}$ that has two sorts $V$ and $K$.  On $V$, there is a constant symbol $0_{V}$, binary functions $V^{2} \to V$ labelled $+_{V}$, $-_{V}$, $[\cdot,\cdot]_{V}$, and $c+1$ unary predicates $P_{i}$ for $1 \leq i \leq c+1$.  Additionally, for each $n$, we have an $n$-ary relation symbol $\theta_{n}$ on $V$.  On $K$, there is the language of rings: constant symbols $0_{K}$, $1_{K}$, and binary functions $K^{2} \to K$ labelled $+_{K}$, $-_{K}$, and $\times_{K}$.  Additionally, there is a function symbol $\cdot : K \times V \to V$ and for each $n$ and $1 \leq i \leq n$, a function symbol $\pi_{n,i} : V^{n+1} \to K$.  We write $\LL_{K,V,c}^{-}$ for the language $\LL_{K,V,c}$ with the binary function $[\cdot,\cdot]$ removed.  We will usually omit the subscripts on the function symbols on $V$ and $K$, when they are understood from context.

We define an $\LL_{K,V,c}$ theory $T_{0}$ via the following axioms.
\begin{enumerate}
    \item $K$ is a field.
    \item $V$ is a $K$-vector space, with scalar multiplication defined by the function $\cdot : K \times V \to V$.  
    \item $[\cdot,\cdot]:V^{2} \to V$ is an alternating $K$-bilinear map satisfying the Jacobi identity, i.e. defines a Lie algebra on $V$.  
    \item Each $P_{i}$ is a vector subspace of $V$ such that $(P_{i}(V))_{1 \leq i \leq c+1}$ forms a Lazard series for the Lie algebra $(V,[\cdot,\cdot])$, that is, we have  
    $$
    V = P_{1}(V) \supseteq P_{2}(V) \supseteq \ldots \supseteq P_{c}(V) \supseteq P_{c+1}(V) = 0,
    $$
    and, additionally, 
    $$
    [P_{i}(V),P_{j}(V)] \subseteq P_{i+j}(V)
    $$
    for all $1 \leq i,j \leq c+1$ (where $P_{i+j}(V)$ is understood to be the trivial subspace for all $i+j > c$).  
    \item The relation symbol $\theta_{n}(v_{1}, \ldots, v_{n})$ holds if and only if $v_{1}, \ldots, v_{n}$ are linearly independent.  If $v_{1}, \ldots, v_{n}$ are linearly independent and $w$ is in their span, then $\pi_{n,i}(v_{1}, \ldots, v_{n},w) = \alpha_{i}$ where
    $$
    w = \sum_{j = 1}^{n} \alpha_{j}v_{j}. 
    $$
    If the $v_{1}, \ldots, v_{n}$ are not linearly independent, or if $w$ is not in their span, we set $\pi_{n,i}(v_{1}, \ldots, v_{n},w) = 0$. 
 We will refer to the functions $\pi_{n,i}$ as \emph{coordinate functions}. 
\end{enumerate}
Our goal is to axiomatize and study the model companion $T$ of $T_{0}$. See \cite[Section 2.2]{dobrowolskisets} for a discussion of why the coordinate functions are necessary for quantifier elimination, even in the reduct to the underlying two-sorted vector space.  Note, however, that if $M \models T_{0}$ and $v_{1}, \ldots, v_{n} \in V(M)$, it makes sense to say $v_{1}, \ldots, v_{n}$ are linearly independent, without specifying the field: since linear independence in a model $M \models T_{0}$ is equivalent to $M \models \theta_{n}(v_{1}, \ldots, v_{n})$, linear independence will be preserved in extensions to larger models of $T_{0}$. 

It will turn out that, in the model-companion of $T_0$, the field will be algebraically closed, since the model companion of the theory of all fields is ACF. More generally, we could consider Lie algebras over a field $K$ which is not algebraically closed.  If $\LL^{\dagger} \supseteq \LL_{\mathrm{rings}}$ is a language and $T^{\dagger}$ is an $\LL^{\dagger}$-theory extending the theory of fields, we define $T_{0}^{+}$ to be the theory (in the language $\LL_{K,V,c}^+$ which is $\LL_{K,V,c}$ together with $\LL^{\dagger}$ on the field sort) which extends $T_{0}$ with axioms asserting $K \models T^{\dagger}$. 

\subsection{Extension of Scalars} 

If $V$ is a vector space over a field $K$ and $K'/K$ is a field extension, then the \emph{extension of scalars} of $V$ to $K'$ is the $K'$-vector space $K' \otimes_{K} V$.  A set of vectors in $V$ which are linearly independent (over $K$) will remain linearly independent in $K' \otimes_{K} V$ (over $K'$) and the dimension of $K' \otimes_{K} V$ over $K'$ is the same as the dimension of $V$ over $K$.  Any $K$-linear structure on $V$ naturally extends to $K'$-linear structure on $K' \otimes_{K} V$.  In particular, if $[\cdot,\cdot] :V^{2} \to V$ is a Lie bracket on $V$, then we obtain a Lie bracket on $K' \otimes_{K} V$, which we may define on a basis by the following formula:
$$
[\alpha \otimes v, \alpha' \otimes v'] = (\alpha \alpha') \otimes [v,v']
$$
for all $\alpha, \alpha' \in K'$ and $v,v' \in V$.  

\begin{lem} \label{extension of scalars}
    Suppose $M =(K,V) \subseteq N=(K',V')$ are models of $T_{0}$ (or models of $T_{0}^{+}$ for a given theory of fields $T^{\dagger}$).  Then, in any model of $T_0$ containing $N$, we have $\langle K',V\rangle \cong (K', K' \otimes_{K} V)$, where $K' \otimes_{K} V$ denotes the $K'$-vector space obtained from $V$ by extension of scalars. 
\end{lem}

\begin{proof}
    We prove by induction on terms of $\LL_{K,V,c}$ that if $\overline{\alpha}$ is a tuple of scalars from $K'$ and $\overline{v}$ is a tuple of vectors from $V$, then $t(\overline{\alpha}, \overline{v}) \in K'$ or $t(\overline{\alpha},\overline{v}) = \sum_{i < k} \beta_{i} w_{i}$ for $\beta_{i} \in K'$ and $w_{i} \in V$ for some $k$. Since the expression $\beta w$ in $V'$ can be identified with $\beta \otimes w$ in $K' \otimes_{K} V$, this suffices. The desired conclusion is clear for terms $t$ consisting of variables and constants, so the base case is true. It is also clear that the conclusion is preserved by application of the field operations (or, more generally, the operations in $L^{\dagger}$) and the vector space operations.  Bilinearity of the bracket entails that 
    $$
    \left[\sum_{i < k}\beta_{i} w_{i}, \sum_{j < k'}\beta'_{j}w'_{j}\right] = \sum_{\substack{i < k\\j < k'}} (\beta_{i} \beta'_{j})[w_{i},w'_{j}]
    $$
    where $\beta_{i} \beta'_{j} \in K'$ and $[w_{i},w'_{j}] \in V$, so the conclusion is also preserved by the bracket. Finally, we check the coordinate functions.  We will proceed as in \cite[Lemma 1.5]{nickandlotte}. Suppose we are given $v_{1}, \ldots, v_{n+1}$ of the form 
    $$
    v_{i} = \sum_{j = 1}^{m} \alpha_{i,j} w_{j}
    $$
    for vectors $w_{1}, \ldots, w_{m} \in V$ and $\alpha_{i,j} \in K'$. Note that, as some $\alpha_{i,j}$ are allowed to be zero, we may assume that the vectors $w_{j}$ that appear in the above expression are the same for all $v_{i}$. We assume $v_{1}, \ldots, v_{n}$ are linearly independent and $v_{n+1}$ is in their span, so $v_{n+1} = \sum_{i = 1}^{n} \lambda_{i} v_{i}$. Thus, replacing the $w_{j}$'s with a subset, we may assume that the $w_{j}$ are linearly independent. Further, after possibly extending the set of $v_i$'s by linearly independent $w_j$'s, we may consider $m = n$. As $\lambda_{i} = \pi_{n,i}(v_{1}, \ldots, v_{n+1})$, we must show that each $\lambda_{i} \in K'$.  Writing matrices with respect to the basis $w_{1}, \ldots, w_{n}$, we have 
    $$
\left[ 
\begin{matrix}
    \alpha_{1,1} & \ldots & \alpha_{n,1} \\
    \vdots & \ddots & \\
    \alpha_{1,n} & \dots & \alpha_{n,n} 
\end{matrix}
\right] \left[
\begin{matrix}
    \lambda_{1}  \\
    \vdots \\
    \lambda_{n}
\end{matrix}
\right] = \left[
\begin{matrix}
    \alpha_{n+1,1}  \\
    \vdots \\
    \alpha_{n+1,n}
\end{matrix}
\right].
$$
Writing $B = (\alpha_{i,j}) \in \mathrm{GL}_{n}(K')$, we see that 
$$
B^{-1} \left[
\begin{matrix}
    \alpha_{n+1,1}  \\
    \vdots \\
    \alpha_{n+1,n}
\end{matrix}
\right] =  \left[
\begin{matrix}
    \lambda_{1}  \\
    \vdots \\
    \lambda_{n}
\end{matrix}
\right],
$$
which shows that indeed  each $\lambda_{i} \in K'$.  
    \end{proof}
%    Suppose $\beta_{1}w_{1}, \ldots, \beta_{n+1} w_{n+1}$ are vectors satisfying $\beta_{i} \in K'$, $w_{i} \in V$, $\{\beta_{i}w_{i} : i = 1, \ldots, n\}$ are linearly independent and 
 %   $$
  %  \beta_{n+1}w_{n+1} = \sum_{i = 1}^{n} \alpha_{i} (\beta_{i} w_{i}). 
   % $$
   % We need to show that $\pi_{n,i}(\beta_{1} w_{1}, \ldots, \beta_{n+1}w_{n+1}) = \alpha_{i} \in K'$. We may assume $\beta_{n+1} \neq 0$. We know that $w_{1}, \ldots, w_{n}$ are also linearly independent and we have 
   % $$
   % \pi_{n,i}(w_{1}, \ldots, w_{n+1}) = \frac{\alpha_{i}\beta_{i}}%{\beta_{n+1}},
%    $$
 %   and since the $w_{j}$ are in $M$, we have $\pi_{n,i}(w_{1}, \ldots, w_{n+1})$ is also in $M$ and therefore in $K$. This shows $\alpha_{i} \in K'$. 

\subsection{Structure constants}
We write $[n]$ to denote the set $\{1, \ldots, n\}$.  Suppose $K$ is a field.  If $v_{1}, \ldots, v_{n}$ is a basis of a Lie algebra $(L,+,[\cdot,\cdot])$ over $K$, the \emph{structure constants} $(\alpha_{i,j,k})_{i,j,k \in [n]}$ are scalars from $K$ that express the Lie bracket on $L$ in terms of this basis.  More precisely, they are chosen so that 
$$
[v_{i},v_{j}] = \sum_{k = 1}^{n} \alpha_{i,j,k} v_{k}
$$
for all $i,j \in [n]$. The fact that the bracket is alternating is equivalent to the fact that $\alpha_{i,j,k} = -\alpha_{j,i,k}$ for all $i,j,k \in [n]$.  Additionally, the Jacobi identity is equivalent to the equality 
$$
\sum_{l = 1}^{n} \alpha_{j,k,l}\alpha_{i,l,m} + \alpha_{i,j,l} \alpha_{k,l,m} + \alpha_{k,i,l} \alpha_{j,l,m} = 0,
$$
for all $i,j,k,m \in [n]$. Any sequence of scalars satisfying these two conditions will define a Lie bracket on $L$.  

In general, given a sequence $n = k_{1} \geq k_{2} \geq \ldots \geq k_{c} \geq k_{c+1} = 0$, we want to know what properties of the structure constants would define a $c$-nilpotent Lie algebra with basis $v_{1}, \ldots, v_{n}$ and Lazard series $(P_{i})_{1 \leq i \leq c+1}$ defined by $P_{i} = \mathrm{Span}(\{v_{l} : 1 \leq l \leq k_{i}\})$.  The desired condition that $[P_{i},P_{j}] \subseteq P_{i+j}$ is equivalent to saying that if $l \leq k_{i}$ and $m \leq k_{j}$, then 
$$
[v_{l},v_{m}] = \sum_{k} \alpha_{l,m,k} v_{k}
$$
where $\alpha_{l,m,k} = 0$ for all $k > k_{i+j}$.  

To summarize, given $n = k_{1} \geq k_{2} \geq \ldots \geq k_{c} \geq k_{c+1} = 0$, we say that a sequence of scalars $(\alpha_{i,j,k})_{i,j,k \in [n]}$ is a sequence of structure constants for a $c$-nilpotent Lazard Lie algebra over $K$ with dimension $k_{i}$ for the $i$th term of the Lazard series, if it satisfies the following three conditions:
\begin{enumerate}
\item $\alpha_{i,j,k} = - \alpha_{j,i,k}$ for all $i,j,k \in [n]$.
\item $\sum_{l = 1}^{n} \alpha_{j,k,l}\alpha_{i,l,m} + \alpha_{i,j,l} \alpha_{k,l,m} + \alpha_{k,i,l} \alpha_{j,l,m} = 0$, for all $i,j,k,m \in [n]$.
\item $\alpha_{i,j,k} = 0$ for all $i,j \in [n]$ and $k > k_{i+j}$.
\end{enumerate}
Condition (3) holding depends on the order in which the vectors $v_{1}, \ldots, v_{n}$ are enumerated, so we say more generally that $(\alpha_{i,j,k})_{i,j,k \in [n]}$ is a sequence of structure constants for a $c$-nilpotent Lazard Lie algebra over $K$ with dimension $k_{i}$ of the $i$th term of the Lazard series if there's some re-indexing of the vectors that satisfies (1), (2), and (3), i.e. if there is some $\sigma \in S_{n}$ such that $(\alpha_{i,j,\sigma(k)})_{i,j,k \in [n]}$ satisfies (1), (2), and (3).  Writing $\overline{k} = (k_{1}, \ldots, k_{c+1})$, we obtain from the above a formula $\varphi_{\mathrm{str},n,\overline{k}}(x_{1}, \ldots, x_{n^{3}})$ which asserts that $(x_{1}, \ldots, x_{n^{3}})$ are the structure constants for a $c$-nilpotent Lazard Lie algebra over $K$ with $\overline{k}$ forming the dimensions of the terms of the Lazard series.  Then we can define $\varphi_{\mathrm{str},n}(x_{1},\ldots, x_{n^3}) = \bigvee_{\overline{k}} \varphi_{\mathrm{str},n,\overline{k}}(x_{1}, \ldots, x_{n^3})$, which   asserts that $(x_{1}, \ldots, x_{n^{3}})$ are the structure constants for a $c$-nilpotent Lazard Lie algebra over $K$.

For each $n$, we have a formula $\varphi_{\mathrm{Lie},n}(x_{1}, \ldots, x_{n})$, where each $x_{i}$ is in the vector space sort, which asserts that $x_{1}, \ldots, x_{n}$ form a basis of Lie subalgebra of $(V,[\cdot,\cdot])$.  To say this, it is necessary to say that $x_{1}, \ldots, x_{n}$ are linearly independent and that $[x_{i},x_{j}]$ is in the span of $x_{1}, \ldots, x_{n}$ for all $i,j \in [n]$. If $v_{1}, \ldots, v_{n}$ are the basis of a Lie subalgebra, then there is a definable function that associates to $v_{1}, \ldots, v_{n}$ their $n^{3}$ structure constants in $K$.  More precisely, we define 
$$
\mathrm{Str}_{n}(v_{1}, \ldots, v_{n}) = (\pi_{n,k}(v_{1}, \ldots, v_{n},[v_{i},v_{j}]))_{i,j,k \in [n]},
$$
which are the structure constants for the Lie subalgebra spanned by $v_{1}, \ldots, v_{n}$. 

\section{The model companion and elimination of quantifiers} \label{axiom subsection}

\subsection{A quantifier elimination result}
We define an $\LL_{K,V,c}$-theory $T$ that extends $T_{0}$ with the following axiom scheme:
\begin{enumerate}
    \item $K \models ACF$ and $V$ is infinite dimensional. 
    \item For all $n$, we have the axiom
    $$
    (\forall \overline{\alpha} \in K^{n^{3}})[\varphi_{\mathrm{str},n}(\overline{\alpha}) \to (\exists \overline{v} \in V^{n})(\varphi_{\mathrm{Lie},n}(\overline{v}) \wedge \mathrm{Str}_{n}(\overline{v}) = \overline{\alpha})],
    $$
    which asserts that every tuple of scalars that define structure constants is the tuple of structure constants of a Lie subalgebra of $(V,[,])$. 
    \item For all $n, m \geq 1$, we have the axiom 
    $$
    (\forall \overline{v} \in V^{n})(\forall \overline{\alpha} \in K^{(n+m)^{3}})\left[(\varphi_{\mathrm{str},(n+m)}(\overline{\alpha}) \wedge \varphi_{\mathrm{Lie},n}(\overline{v}) \wedge \mathrm{Str}_{n}(\overline{v}) \subseteq \overline{\alpha}) \to (\exists \overline{w} \in V^{m})[\mathrm{Str}_{n+m}(\overline{v},\overline{w}) = \overline{\alpha}]\right]
    $$
    which asserts that if $\overline{v} \in V^{n}$ is the basis of a Lie subalgebra of $V$ and $\overline{\alpha}$ is a tuple of scalars that defines structure constants and which contains the structure constants of $\ol v$, then there is some $\overline{w}$ such that $(\overline{v},\overline{w})$ is the basis of a Lie subalgebra with structure constants $\overline{\alpha}$. 
\end{enumerate}

We will show that $T$ is the model companion of $T_{0}$.  When we are working with Lie algebras over a field satisfying the theory $T^{\dagger}$, then $T_{0}^{+}$ and $T^{+}$ are defined analogously but with adding the condition $K \models T^{\dagger}$ (replacing axiom (1) in $T$). We will sometimes write $T^+ = T^+(T^\dagger)$ to emphasize the dependence of $T^+$ on $T^\dagger$. 
%{\color{blue}Observe that axioms (2) and (3) of $T^+$ do not depend on $T^\dagger$, hence using ??? we immediately get:
%
%
%\begin{lemma}
 %   Let $(K,V)\models T^+$ and $F\seq K$ be a subfield of $K$. Let $\LL$ be a language extending the language of fields and assume that $F$ is endowed with some $\LL$-structure. Then the model $(F,V)$ of $T_0$ is a model of $T^+ = T^+(\Th_\LL(F)))$.
%\end{lemma}
%}

\begin{lem}
    Every model of $T_{0}$ extends to a model of $T$.  Additionally, every model of $T_{0}^{+}$ extends to a model of $T^{+}$.  
\end{lem}

\begin{proof}
 If $M = (K(M), V(M), [\cdot,\cdot]^{M}) \models T_{0}$, then, by extension of scalars, we can assume $K(M)$ is an algebraically closed field. Hence $T = T^{+}$ for the special case that $\LL^{\dagger} = \LL_{\mathrm{rings}}$ and $T^{\dagger} = ACF$.  Therefore, it suffices to argue that if $M = (K(M), V(M), [\cdot,\cdot]^{M}) \models T_{0}^{+}$, then $M$ embeds into a model of $T^{+}$.  
 
 Let $(\overline{\alpha}_{i})_{i}$ list all possible finite sequences of scalars from $K(M)$ which are the structure constants of a finite dimensional $c$-nilpotent LLA $L_{i}$ over $K(M)$.  Let $V' = V \oplus \bigoplus_{i} L_{i}$ and let $[\cdot,\cdot]' : V' \times V' \to V'$ be the associated direct sum Lie bracket.  Then $M' = (K(M), V', [\cdot,\cdot]')$ satisfies the axiom schema (1) and (2) in the definition of $T^{+}$.  

    To satisfy the axiom scheme (3), we will build the model as the union of a chain.  Let $M_{0} = M'$ be the model as constructed above. Now, given $M_{i}$, we will construct an extension $M_{i+1}$ such that if $\overline{v}$ is a basis of a finite dimensional subalgebra of $(V(M_{i}),[\cdot,\cdot]^{M_{i}})$ and $\overline{\alpha}$ is a finite sequence of structure constants for some $c$-nilpotent LLA extending the subalgebra spanned by $\overline{v}$, then there is some $\overline{w}$ in $V(M_{i+1})$ such that $\overline{v}$ and $\overline{w}$ together form a basis for a subalgebra of $(V(M_{i+1}),[\cdot,\cdot]^{M_{i+1}})$ with structure constants $\overline{\alpha}$. Let $(L_{j},\overline{\alpha}_{j})_{j < \lambda}$ list all pairs where $L_{j}$ is a finite dimensional subalgebra of $(V(M_{i}),[\cdot,\cdot]^{M_{i}})$ and $\overline{\alpha}_j$ is a finite sequence of structure constants for some $c$-nilpotent LLA extending $L_{j}$.  For each $j$, let $L'_{j}$ be an LLA extending $L_{j}$ with structure constants $\overline{\alpha}_{j}$.  
    
    We define $N_{0}$ to be the LLA $(V(M_{i}), [\cdot,\cdot]^{M_{i}})$.  Then given $N_{j}$ for some $j < \lambda$, we define $N_{j+1}$ to be the free amalgam of $N_{j}$ and $L'_{j}$ over $L_{j}$, as LLAs over $K(M_{i})$. Note that this exists by \cite[Theorem 4.35]{d2023model} (where it is denoted $N_{j} \otimes_{L_{j}} L'_{j}$, viewing them as LLAs over $K(M_{i})$).  We view $N_{j+1}$ as an extension of $N_{j}$.  As for limit ordinals $\delta \leq \lambda$, we define $N_{\delta} = \bigcup_{j < \delta} N_{j}$. Then we set $M_{i+1} = (K(M_{i}), N_{\lambda}, [\cdot,\cdot]^{N_{\lambda}})$. 

    To conclude we set $M'' = \bigcup_{i < \omega} M_{i}$. Axiom scheme (1) is clearly satisfied.  Moreover, since $M'$ satisfies the second axiom scheme and $M''$ has the same set of scalars, $M''$ satisfies it as well.  Finally, if $L$ is a finite dimensional subalgebra of $M''$ and $\overline{\alpha}$ is a sequence of structure constants for a finite dimensional extension of $L$, then $L$ is contained in $M_{i}$ for some $i$ and hence a extension of $L$ with structure constants $\overline{\alpha}$ can be found as a subalgebra of $M_{i+1}$, hence of $M''$. This shows the axiom scheme (3) is satisfied as well, completing the proof. 
\end{proof}

\begin{lem}
    The theory $T$ has quantifier elimination and its completions are determined by specifying the characteristic of the field.  More generally, if $T^{\dagger}$ is a complete theory extending the theory of fields with quantifier elimination, then $T^{+}$ is complete and has quantifier-elimination.  
\end{lem}

\begin{proof}
We first argue for $T$.  We use the well-known criterion for QE that if $M$ and $N$ are two $\aleph_{0}$-saturated models of $T$ whose fields have the same characteristic, then the set of partial isomorphisms from a finitely generated substructure of $M$ to a finitely generated substructure of $N$ has the back-and-forth property. Non-emptiness is easy, since if $K_{0}$ is the prime subfield then the unique map $f:(K_{0},\set{0})\to (K_{0},\set{0})$ is a partial isomorphism. 

Now suppose $M_{0} \subseteq M$ and $N_{0} \subseteq N$ are finitely generated substructures and $f : M_{0} \to N_{0}$ is an isomorphism. Suppose $a \in M \setminus M_{0}$.  Let $M_{1} = \langle a M_{0} \rangle$. We have to consider three cases:

\textbf{Case 1}:  $a \in K(M)$.  

Because $K(N)$ is algebraically closed, we can find some $b$ such that the isomorphism from $K(M_{0})$ to $K(N_{0})$ given by $f$ extends to an isomorphism from $K(M_{1})$, i.e. the subfield of $K(M)$ generated by $K(M_{0})a$, to the subfield of $K(N)$ generated by $K(N_{0})b$ and mapping $a \mapsto b$.  Let $N_{1} = \langle N_{0},b \rangle$.  Then, by Lemma \ref{extension of scalars}, $(V(M_{1}), [\cdot,\cdot])$ and $(V(N_{1}), [\cdot,\cdot])$ are obtained by extension of scalars from $K(M_{0})$ to $K(M_{1})$ and from $K(N_{0})$ to $K(N_{1})$ respectively.  The isomorphism $K(M_{1})$ to $K(N_{1})$ extending $f$, then, induces an isomorphism from $V(M_{1})$ to $V(N_{1})$ respecting the Lie algebra structure, as desired. 
%Note:  we're using here that extending the field only results in extension by scalars; this probably should be a separate lemma just for completeness.

\textbf{Case 2}: $a \in V(M)$ and $a$ is in the $K(M)$-span of $V(M_{0})$.  

Choose a basis $\overline{v}$ for $V(M_{0})$.  If $a$ is in the $K(M)$-span of $\overline{v} = (v_{0}, \ldots, v_{n-1})$ then, since $a \not\in V(M_{0})$, it must be the case that 
$$
a = \sum_{i < n} \alpha_{i} v_{i}
$$
with not all $\alpha_{i} \in K(M_{0})$.  Applying Case 1 at most $n$ times, we may extend $f$ to an isomorphism $f': \langle M_{0}, \alpha_{<n} \rangle \to \langle N_{0}, \beta_{<n} \rangle$ where $f'(\alpha_{i}) = \beta_{i}$ for all $i < n$.  Note that then
$$
f'(a) = \sum_{i < n} \beta_{i}f(v_{i}).
$$
Because we have the coordinate functions in the language, we have $M_{1} = \langle M_{0}, \alpha_{<n}\rangle$ and, setting $b = f'(a)$, we also have $N_{1} = \langle N_{0},b \rangle = \langle N_{0}, \beta_{<i} \rangle$ and $f' :M_{1} \to N_{1}$ is the desired extension. 

\textbf{Case 3}:  $a \in V(M)$ and $a$ is not in the $K(M)$-span of $V(M_{0})$.

We may assume that $[a,u]$ is in the $K(M)$-span of $V(M_{0})$ for all $u \in V(M_{0})$ (i.e that the Lie algebra over $K(M)$ spanned by $V(M_{0})$ is an ideal of the Lie algebra over $K(M)$ spanned $V(M_{1})$) since the general case follows from this one by reverse induction on the maximum $1\leq i\leq c+1$ such that $a\in P_i(M)\setminus P_{i+1}(M)$. Let $\{v_1,\dots, v_n\}$ be a $K(M_0)$-basis of $M_0$ and let $M_{0}'$ be the substructure of $M$ generated by $M_{0}$ and $\{[a,v_{i}] : i < n\}$.  Notice that $\overline{v}$ is still a basis of $V(M_{0}')$ though the field may grow. 
 By at most $n$ applications of Case 2, there is a structure $N_{0} \subseteq N_{0}' \subseteq N$ and an isomorphism $f': M_{0}' \to N_{0}'$ extending $f$.  

 Let $\overline{\alpha}$ be the sequence of structure constants for the subalgebra of $M$ spanned by $\overline{v}$ and $a$. By construction, $\overline{\alpha}$ is contained in $K(M_{0}')$ and hence $f'(\overline{\alpha})$ is contained in $K(N_{0}')$.  Since $f'(\overline{\alpha})$ is a sequence of structure constants for an LLA extending $f'(\overline{v})$, the axioms of $T$ entail that there is some $b \in V(N)$ such that $f'(\overline{v})$ and $b$ span an LLA with structure constants $f'(\overline{\alpha})$.  Then the map extending $f'$ and mapping $a \mapsto b$ defines an isomorphism from $M_{1} = \langle M_{0}'a \rangle$ to $\langle N_{0}'b \rangle$, which gives the desired extension.  

 The proof in the case that we are considering $T^{+}$ is essentially identical, with minor modifications. %First, to argue that the set of isomorphisms between finitely generated substructures is non-empty, we have that, for vectors $v \in V(M)$ and $w \in V(N)$, $\langle v \rangle$ and $\langle w \rangle$ are isomorphic as Lie algebras over the substructure of $K$ generated by $\emptyset$, which might not be the prime subfield \red{again $\set{0}\to \set{0}$ suffices?}. 
 Instead of considering all partial isomorphisms, we only consider those partial isomorphisms in which the induced $L^\dagger$-isomorphism of fields $K(M_0)\to K(N_0)$ is $T^\dagger$-elementary.  Then, since $K(M)$ and $K(N)$ are $\aleph_{0}$-saturated as models of $T^{\dagger}$, this map may be extended to incorporate a new field element, which yields Case 1. Case 2 and 3 are identical.
\end{proof}

    The preceding proof yields the following result: 
        \begin{center}
            $(\star)\quad$ the family of $\LL_{K,V,c}$-isomorphisms between substructures of models of $T^+$ such that the restriction to the field sort is $\LL^\dagger$-elementary has the back and forth property. 
        \end{center} This has the following important consequence for $T^{+}$:  if $A$ and $B$ are substructures of a monster model $\mathbb{M} \models T^{+}$ which contain a common substructure $C$, if $K(A) \equiv_{K(C)} K(B)$ in $T^{\dagger}$ and there is an isomorphism $A \to B$ over $C$, then $A \equiv_{C} B$.  This says that $T^{+}$ eliminates quantifiers relative to the field sort.  This is recorded in the following corollary:

\begin{corollary}\label{cor:typedecomposition}
    Let $(K,V)$ be a model of $T^+$. Then for all tuples $\bar \alpha,\bar \beta$ from $K$ and $\bar a,\bar b$ such that $\vect{\bar \alpha,\bar a} = (\bar \alpha, \bar a)$ and $\vect{\bar \beta, \bar b} = (\bar \beta, \bar b)$ (i.e. $\overline{\alpha}\overline{a}$ and $\overline{\beta} \overline{b}$ enumerate substructures) from $V$ we have 
    \[\bar \alpha \equiv^{\LL^\dagger} \bar \beta \text{ and } (\bar \alpha,\bar a) \equiv^{qf,\LL_{K,V,c}} (\bar \beta,\bar b) \iff (\bar \alpha,\bar a) \equiv^{\LL^+} (\bar \beta,\bar b)\]
\end{corollary}
\begin{proof}
    Immediately follows from $(\star)$.
\end{proof}

\begin{cor}
 $T$ is the model completion of $T_{0}$.
\end{cor}

\begin{proof}
    Immediate. 
\end{proof}

Our quantifier elimination result can be slightly upgraded, in a way that will be especially useful in our later proof that $T^{+}$ is NSOP$_{4}$.  

\begin{lem} \label{lem:better base}
    Suppose $A_{i} = \langle A_{i},B \rangle$ for $i = 0,1$ and $A_{0} \equiv_{B} A_{1}$, where $B$ is a substructure of $\mathbb{M}$. Suppose $K(B) \subseteq K' \subseteq K(\mathbb{M})$ and $K(A_{0}) \equiv_{K'} K(A_{1})$.  Then $A_{0} \equiv_{BK'} A_{1}$. 
\end{lem}

\begin{proof}
    Let $f : A_{0} \to A_{1}$ be a partial elementary isomorphism fixing $B$ pointwise. So we may write $f = (f_{K},f_{V})$ where $f_{K} : K(A_{0}) \to K(A_{1})$ fixes $K(B)$ and $f_{V}: V(A_0) \to V(A_1)$ fixes $V(B)$. Since $K(A_{0}) \equiv_{K'} K(A_{1})$, we know $f_{K}$ extends to a partial elementary map $g: K(A_{0})K' \to K(A_{1})K'$ which fixes $K'$ pointwise. Note that by Lemma \ref{extension of scalars}, $\langle A_{i},K'\rangle = ((K(A_{i})K'),K(A_{i})K' \otimes_{K(A_{i})} V(A_{i}))$ for $i = 0,1$. We can define an isomorphism $h : \langle A_{0},K' \rangle \to \langle A_{1},K' \rangle$ with $h = (h_{K},h_{V})$ where $h_{K} = g$ and $h_{V} : (K(A_{0})K') \otimes_{K(A_{0})} V(A_{0}) \to (K(A_{1})K') \otimes_{K(A_{1})} V(A_{1})$ by setting 
    $$
    h(\alpha \otimes v) = g(\alpha) \otimes f_{V}(v)
    $$
    for all $\alpha \in K(A_{0})K'$ and $v \in V(A)$ and extending linearly. Since $h_{K} = g$ is partial elementary in the field sort, this isomorphism is partial elementary.  Since $h$ fixes $B$ (as it extends $f$) and fixes $K'$ (by the choice of $g$), the map $h$ witnesses $A_{0} \equiv_{BK'} A_{1}$. 
\end{proof}

We make one additional observation about back-and-forth arguments between LLAs. The following is true for any theory $T^{\dagger}$ of fields.  Although when the field is infinite, $T^{+}$ will not be $\aleph_{0}$-categorical, it is nonetheless $\aleph_{0}$-categorical `relative to the field' in the following sense:

\begin{obs}
Suppose $M = (K,V), M' = (K',V') \models T^{+}$ are countable models and $K \cong K'$.  Then $M \cong M'$. 
\end{obs}

\begin{proof}
An easy back-and-forth. 
\end{proof}

Our quantifier elimination also gives us a description of algebraic closure:

\begin{lem} \label{lm:algclos}
    Let $M \models T^{+}$. A substructure $A \subseteq M$ is algebraically closed if and only if $K(A) = \mathrm{acl}_{L^{\dagger}}(K(A))$.  Hence, for an arbitrary subset $X \subseteq M$, we have
    $$
    \mathrm{acl}(X) = \langle \mathrm{acl}_{L^{\dagger}}(\langle X \rangle),X \rangle. 
    $$
\end{lem}

\begin{proof}
    We may replace $M$ with a very saturated elementary extension $\mathbb{M}$. Suppose $A$ is a substructure of $\mathbb{M}$ with $\mathrm{acl}_{L^{\dagger}}(K(A)) = K(A)$ and let $B$ be a substructure with $A \subseteq B \subseteq \mathbb{M}$.  Let $(K_{i})_{i < \omega}$ be a $K(A)$-indiscernible seqeunce in $K(\mathbb{M})$ with $K_{0} = K(B)$ and $K_{i} \ind^{a}_{K(A)} K_{<i}$ for all $i < \omega$, where $\ind^{a}$ is computed in $T^{\dagger}$. Note that, by quantifier-elimination, we have that $(K_{i})_{i < \omega}$ is $A$-indiscernible in $\mathbb{M}$. Let $B_{0} = B$ and choose, for each $i > 0$, some $B_{i}$ with $K(B_{0}) B_{0} \equiv_{A} K_{i} B_{i}$. In particular, we have $K(B_{i}) = K_{i}$ for all $i$.  

    Let $\tilde{K} = K(\langle B_{i} : i < \omega \rangle)$. Let $\tilde{A}$ and $\tilde{B}_{i}$ denote the respective extensions of scalars of $A$ and $B_{i}$ to $\tilde{K}$\textemdash that is, $\tilde{A} = \langle \tilde{K}, A\rangle$ and $\tilde{B}_{i} = \langle \tilde{K}, B_{i} \rangle$ for all $i < \omega$. Then, inductively, we define a sequence $(\tilde{B}'_{i})_{i < \omega}$ by setting $\tilde{B}'_{0} = \tilde{B}_{0}$ and, given $\tilde{B}_{0}, \ldots, \tilde{B}_{i}$, we choose $\tilde{B}_{i+1}' \equiv_{\tilde{A}} \tilde{B}_{i+1}$ with $\tilde{B}'_{i+1} \ind^{\otimes}_{\tilde{A}} \tilde{B}'_{\leq i}$ as LLAs over $\tilde{K}$. Choose, for each $i > 0$, some $\sigma_{i} \in \mathrm{Aut}(\mathbb{M}/\tilde{A})$ such that $\sigma(\tilde{B}_{i}) = \tilde{B}'_{i}$ and define $B'_{i} = \sigma(B_{i})$.  Since each $\sigma_{i}$ fixes $\tilde{A}$, it fixes $\tilde{K}$ and thus fixes $K(B_{i})$.  It follows that $K(B'_{i}) = K(B_{i})$.  We also have, by free amalgamation, that $V(B_{i})$ is linearly independent from $V(B'_{\leq i})$ over $V(A)$.  By quantifier elimination, we have $B'_{i} \equiv_{A} B$ and, by construction, $B'_{i} \cap B'_{j} = A$ for all $i \neq j$. This shows that $A$ is algebraically closed.  
\end{proof}

% \begin{cor}
%     If  $T^\dagger$ is a complete theory extending the theory of fields without quantifier elimination, that $T^+$ still admits elimination of quantifiers down to the field in the following sense: If $p(x,y)$ is any complete type, where $xy$ enumerate a structure with $x$ naming the vector space variables and $y$ the field variables and $q(y)$ is the restriction of $p(x,y)$ to $y$, then
%     $$ p_{qf}(x,y) \cup q(y) \vdash p(x,y)$$
% \end{cor}

\subsection{The asymptotic theory of the Fra\"iss\'e limit of nil-$c$ LLAs over $\mathbb{F}_{p}$}

Recall we write $\mathbf{L}_{c,p}$ to denote the Fra\"iss\'e limit of $c$-nilpotent Lazard Lie algebras over $\mathbb{F}_{p}$, which was shown to exist in \cite[Theorem 4.37]{d2023model}.  The Lie algebra $\mathbf{L}_{c,p}$ may be naturally viewed as an $L_{K,V,c}$-structure, in which $\mathbf{L}_{c,p}$ is the interpretation of $V$ and $K$ is interpreted as the field $\mathbb{F}_{p}$. Call this structure $M_{p}$. We will show, then, that ultraproducts yield models of the the theory of generic $c$-nilpotent Lie algebras over pseudo-finite fields.  Letting $T^{\dagger}$ be the theory of pseudo-finite fields (in $\LL_{\mathrm{rings}}$), then for the associated theory $T^{+}$ of generic $c$-nilpotent Lie algebras over a pseudo-finite field, we have the following:

\begin{thm}
    If $\mathcal{D}$ is a non-principal ultrafilter on the set of primes, then 
    $$
    \prod_{p} M_{p}/\mathcal{D} \models T^{+}. 
    $$
\end{thm}

\begin{proof}
Let $\tilde{M}$ denote the ultraproduct $\prod M_{p}/\mathcal{D}$.  We have that $\mathbf{L}_{c,p}$ and $M_{p}$ are two ways of encoding the same Lie algebra (though $M_{p}$ has as its underlying set the disjoint union of the Lie algebra $\mathbf{L}_{c,p}$ and the field $\mathbb{F}_{p}$).  For each $k\in \N$, and each $k$-generated LLA $A$ over $\mathbb{F}_{p}$, there is a formula $\varphi_{A}(x_{1}, \ldots, x_{k})$ such that, for any arbitrary $N_{p}$, $N_{p} \models \varphi(a_{1}, \ldots, a_{k})$ if and only if $\langle a_{1}, \ldots, a_{k} \rangle \cong A$.  

Because $\mathbf{L}_{c,p}$ is the Fra\"iss\'e limit of the class of all finite $c$-nilpotent LLAs over $\mathbb{F}_{p}$, the theory $\mathrm{Th}(\mathbf{L}_{c,p})$ is axiomatized by saying the following:
\begin{enumerate}
 \item For each finite $c$-nilpotent LLA $A$ over $\mathbb{F}_{p}$ generated by a basis of $k$ elements, there is an axiom $(\exists x_{1}, \ldots, x_{k})\varphi_{A}(x_{1}, \ldots, x_{k})$.  This expresses that $N_{p}$ has as its age all finite $c$-nilpotent LLAs over $\mathbb{F}_{p}$. 
 \item For each containment $A \subsetneq B$ of finite LLAs over $\mathbb{F}_{p}$, with $A$ $k$-dimensional and $B$ $l$-dimensional, there is an axiom
 $$
 (\forall x_{1}, \ldots, x_{k})(\exists y_{k+1}, \ldots, y_{l})[\varphi_{A}(\overline{x}) \to \varphi_{B}(\overline{x}, \overline{y})].
 $$
\end{enumerate}
 These axioms express that $\mathbf{L}_{c,p}$ satisfies the embedding property and hence is homogeneous.  Now we just need to check that $\tilde{M}$ satisfies the axiom scheme (1)-(3) at the beginning of Subsection \ref{axiom subsection}.  Axiom scheme (1) is clear since $T^{\dagger}$ is the theory of pseudo-finite fields and each $\mathbf{L}_{p,c}$ is clearly infinite dimensional.  Axiom (2) and (3) follow by a standard application of {\L}os's theorem.  As (3) is entirely similar, we will only spell out the details for (2).  
 
 Let $\tilde{K} = K(\tilde{M}) = \prod \mathbb{F}_{p}/\mathcal{D}$. Suppose $\overline{\alpha} = (\alpha_{1},\ldots, \alpha_{n^{3}}) \in \tilde{K}$ and we have 
 $$
 \tilde{M} \models \varphi_{\mathrm{str},n}(\overline{\alpha}).
 $$
 Write $\overline{\alpha} = (\alpha_{1,p}, \ldots, \alpha_{n^{3},p})_{p}/\mathcal{D}$.  For each prime $p$, if $M_p\models \varphi_{\mathrm{str},n}(\alpha_{1,p}, \ldots, \alpha_{n^{3},p})$, we will let $A_{p} \subseteq \mathbf{L}_{p,c}$ be a Lie algebra over $\mathbf{F}_{p}$ with structure constants $(\alpha_{1,p}, \ldots, \alpha_{n^{3},p})$, which exists by the axiom scheme (1) above. Then we can choose $\overline{v}$ from $A_{p}$ such that $(\alpha_{1,p}, \ldots, \alpha_{n^{3},p})$ are structure constants for $A_{p}$ with respect to the basis $v_{1,p}, \ldots, v_{n,p}$.  If $(\alpha_{1,p}, \ldots, \alpha_{n^{3},p})$ are not structure constants, then we set $A_{p} = 0$ (and likewise each $v_{i,p} = 0$).  By {\L}os's theorem, $A_{p} \neq 0$ for all but a $\mathcal{D}$-small set of primes. Setting $(v_{1},\ldots, v_{n}) = (v_{1,p},\ldots, v_{n,p})_{p}/\mathcal{D}$, we have, again by {\L}os's theorem, that 
 $$
 \tilde{M} \models \varphi_{\mathrm{Lie},n}(v_{1},\ldots, v_{n}) \wedge \mathrm{Str}_{n}(\overline{v}) = \overline{\alpha}. 
 $$
 As $\overline{\alpha}$ was an arbitrary $n^{3}$-tuple from $\tilde{K}$, we can conclude that $\tilde{M}$ satisfies axiom scheme (2), completing the proof. 
\end{proof}

\begin{rem}
    Pseudo-finite fields do not eliminate quantifiers in $\LL_{\mathrm{rings}}$ but they do in a reasonable extension of that language.  We could let $\LL^{\dagger} \supseteq \LL_{\mathrm{rings}}$ be the language of rings together with an $(n+1)$-ary relation $\mathrm{Sol}_{n}(x_{0}, \ldots, x_{n})$.  Any finite or pseudo-finite field has a natural expansion to $L^{\dagger}$ in which each relation $\mathrm{Sol}_{n}(x_{0}, \ldots, x_{n})$ is interpreted so that the following holds:
$$
\mathrm{Sol}_{n}(x_{0},\ldots, x_{n}) \leftrightarrow (\exists y)\left(\sum_{i=0}^{n} x_{i}y^{i} = 0\right).
$$
In this language, the theory of pseudo-finite fields eliminates quantifiers (see \cite[Theorem 4.2]{chatzidakis2018notes}). By Corollary \ref{cor:typedecomposition}, $\prod M_{p}/\mathcal{D}$ eliminates quantifiers, when each $M_{p}$ is endowed with an $\LL^{\dagger}$-structure on the field sort. 
\end{rem}

\section{Neostability}
\subsection{NIP$_c$}
First we will show that $T$ is $c$-dependent, deducing this from the Chernikov-Hempel Composition Lemma, as in \cite[Theorem 5.20]{d2023model}: 

\begin{fact} \label{composition lemma} \cite{ChernikovHempel3}
Let $M$ be an $\LL'$-structure such that its reduct to a language $\LL \subseteq \LL'$ is NIP.  Let $d,k$ be natural numbers and $\varphi(x_{1}, \ldots, x_{d})$ be an $\LL$-formula. Let further $y_{0}, \ldots, y_{k}$ be arbitrary $(k+1)$-tuples of variables.  For each $1 \leq t \leq d$, let 
$0 \leq i_{t,1}, \ldots, i_{t,k} \leq k$ be arbitrary and let $f_{t} : M_{y_{i_{t,1}}} \times \ldots \times M_{y_{i_{t,k}}} \to M_{x_{t}}$ be an arbitrary $k$-ary function.  Then the formula 
$$
\psi(y_{0}; y_{1}, \ldots, y_{k}) = \varphi(f_{1}(y_{i_{1,1}}, \ldots, y_{i_{1,k}}), \ldots, f_{d}(y_{i_{d,1}}, \ldots, y_{i_{d,k}}))
$$
is $k$-dependent. 
\end{fact}

We will start by giving a description of terms in this theory, along the lines of \cite[Lemma 5.17]{d2023model} there, though there are new complications due to the extra structure. We recall that $\LL^{-}_{V,K,c}$ is the language $\LL_{V,K,c}$ with the Lie bracket symbol removed. 

\begin{lemma} \label{term description 2}
    Suppose $\overline{x}$ is a tuple of field variables, $\overline{y}$ is a tuple of vector space variables, and $t(\overline{x},\overline{y})$ is a term of $\LL_{K,V,c}$. 
    \begin{enumerate}
        \item If $t$ is valued in the vector space sort, then 
        $$
        t(\overline{x},\overline{y}) = \sum_{i = 1}^{n} t_{i}(\overline{x},\overline{y}) \cdot m_{i}(\overline{y})
        $$
        where each $t_{i}(\overline{x},\overline{y})$ is a term valued in the field sort, and each $m_{i}(\overline{y})$ is a Lie monomial whose variables come from $\overline{y}$.  
        \item If $t$ is valued in the field sort, then $t(\overline{x},\overline{y}) = s(\overline{x},\overline{m}(\overline{y}))$, where $s$ is a term of $\LL_{K,V,c}^{-}$ and $\overline{m}(\overline{y})$ is a tuple of Lie monomials in the variables $\overline{y}$. 
    \end{enumerate}
\end{lemma}

\begin{proof}
(1)  It is clear that the statement is true for variables and constants.  It is also clear that the class of terms of this form is closed under scalar multiplication by a field valued term $s(\overline{x},\overline{y})$, vector addition and subtraction, and by application of the Lie bracket.  Therefore, the conclusion follows by induction on terms.

(2) We will argue by induction on terms using (1).  The conclusion is obvious for constants and variables.  The class of terms satisfying the description in (2) is also clearly closed under the field operations.  The only non-trivial case, then, is to check that a coordinate function applied to vector space sort valued terms can again be written in this form.  Fix natural numbers $i_{*} \leq n$. So assume $t_{1}(\overline{x},\overline{y}), \ldots, t_{n+1}(\overline{x},\overline{y})$ are vector space sort valued terms and 
$$
t(\overline{x},\overline{y}) = \pi_{n,i_{*}}(t_{1}(\overline{x},\overline{y}), \ldots, t_{n+1}(\overline{x},\overline{y}))
$$
is a term of minimal complexity for which the conclusion has not yet been established. By (1), for each $1 \leq i \leq n+1$, we have
$$
t_{i}(\overline{x},\overline{y}) = \sum_{j = 1}^{k_{i}} t_{i,j}(\overline{x},\overline{y})m_{i,j}'(\overline{y}),
$$
where each $t_{i,j}$ is a term valued in the field sort and $m'_{i,j}(\overline{y})$ is a Lie monomial. By induction, then, each term $t_{i,j}(\overline{x},\overline{y}) = s_{i,j}(\overline{x}, \overline{m}_{i,j}(\overline{y}))$, where $s_{i,j}(\overline{x},\overline{z}_{i,j})$ is a term of $\LL_{K,V,c}^{-}$ and $\overline{m}_{i,j}(\overline{y})$ is a tuple of Lie monomials.  Therefore, we have
$$
t(\overline{x},\overline{y}) = \pi_{n,i_{*}}\left( \sum_{j = 1}^{k_{1}} s_{1,j}(\overline{x},\overline{m}_{1,j}(\overline{y}))m'_{1,j}(\overline{y}), \ldots, \sum_{j = 1}^{k_{n+1}} s_{n+1,j} (\overline{x},\overline{m}_{n+1,j}(\overline{y}))m'_{n+1,j}(\overline{y})\right).
$$
Consider the $\LL_{K,V,c}^{-}$ term $s(\overline{x},\overline{z},\overline{w})$ defined by 
$$
s(\overline{x},\overline{z},\overline{w}) = \pi_{n,i_{*}}\left( \sum_{j = 1}^{k_{1}} s_{1,j}(\overline{x},\overline{z}_{1,j})w_{i,j}, \ldots, \sum_{j = 1}^{k_{n+1}} s_{n+1,j} (\overline{x},\overline{z}_{n+1,j})w_{i,j}\right),
$$
where $\overline{z}$ is a tuple concatenating the tuples $\overline{z}_{i,j}$ and $\overline{w}$ is a tuple enumerating the $w_{i,j}$. Then if $\overline{m}(\overline{y})$ concatenates the tuples $\overline{m}_{i,j}(\overline{y})$ and $\overline{m}'(\overline{y})$ enumerates the $m'_{i,j}(\overline{y})$, then we obtain 
$$
t(\overline{x},\overline{y}) = s(\overline{x},\overline{m}(\overline{y}),\overline{m}'(\overline{y})),
$$
which has the desired form.  This completes the induction. 
\end{proof}

\begin{lem} \label{reduct is stable}
    The reduct of $T$ to the language $\LL^{-}_{K,V,c}$ is stable.
\end{lem}

\begin{proof}
 The reduct of $T$ to the language $\LL^{-}_{K,V,c}$ is the theory of an infinite dimensional vector space over an algebraically closed field, with a distinguished flag of length $c$ such that each quotient $P_{i}(M)/P_{i+1}(M)$ is infinite-dimensional. It is complete after specifying a characteristic of the field and a model $M$ of size $\aleph_{1}$ of a completion is determined by specifying the transcendence degree of the field over the prime subfield, and the dimension of each quotient $P_{i}(M)/P_{i+1}(M)$.  There are countably many choices for the transcendence degree, since this can be finite, $\aleph_{0}$, or $\aleph_{1}$, and there are two choices, $\aleph_{0}$ or $\aleph_{1}$, for each quotient. It follows that there are only $\aleph_{0}$ many models of size $\aleph_{1}$.  This implies that every completion of the reduct of $T$ to the language $\LL^{-}_{K,V,c}$ is $\omega$-stable.  
\end{proof}

\begin{thm}
    The theory $T$ is $c$-dependent and $(c-1)$-independent.  
\end{thm}

\begin{proof}
    The proof that $T$ is $(c-1)$-independent is identical to the argument of \cite[Theorem 5.20]{d2023model}, replacing the free $c$-nilpotent Lie algebra over $\mathbb{F}_{p}$ with the free $c$-nilpotent Lie algebra over any algebraically closed field. Moreover, it follows by Lemma \ref{term description 2} that every formula $\varphi(\overline{x},\overline{y})$ of $\LL_{K,V,c}$, where $\overline{x}$ is a tuple of field sort variables and $\overline{y}$ is a tuple of vector space sort variables, can be written in the form 
    $$
    \psi(\overline{x},\overline{m}(\overline{y}))
    $$
    for an $\LL_{K,V,c}^{-}$-formula $\psi(\overline{x},\overline{z})$, where $\overline{m}(\overline{y})$ is a tuple of Lie monomials in the variables $\overline{y}$. By $c$-nilpotence, every Lie monomial is at most $c$-ary.  By Lemma \ref{reduct is stable}, the formula $\psi(\overline{x},\overline{z})$ is stable. Therefore, by Fact \ref{composition lemma}, we conclude that $T$ is $c$-dependent. 
\end{proof}

\begin{rem}
    For the case of $c = 2$, the proof actually gives that $T$ is NFOP$_{2}$ by \cite[Theorem 2.16]{abdaldaim2023higher}.  
\end{rem}

\subsection{Generalities in NSOP$_1$ theories} In this subsection we fix a monster model $\M$ of a fixed theory $T$. In general, $a,b,c,\ldots$ are (possibly infinite) tuples from $\M$;  $A,B,C,\ldots$ are small subsets of $\M$ and $M,N,\ldots$ are small elementary substructures of $\M$.

\begin{defn}
We write $a\indi u_E b$ if $\tp(a/bE)$ is finitely satisfiable in $E$. We write $a\indi h_E b$ if $b\indi u_E a$.
\end{defn}

\begin{defn}
    For any ternary relation $\ind$ over small subsets of $\M$ we call a sequence $(a_i)_{i<\omega}$ an \textit{$\ind$-Morley sequence over $C$} if it is $C$-indiscernible and $a_i\ind_C a_{<i}$ for all $i<\omega$.
\end{defn}

\begin{fact}\label{fact:indiscernibleisheirovermodel}
    Suppose $(a_i)_{i<\omega}$ is an indiscernible sequence in $\M$. Then there is a model $M$ such that $(a_i)_{i<\omega}$ is $\indi h$-independent over $M$.
\end{fact}
\begin{proof}
    Apply \cite[Theorem 3.3.7]{delbee2023axiomatic}. 
\end{proof}

\begin{defn}
Let $M\prec \M$ be a small model. 
\begin{itemize}
    \item A type $p$ over a model $M$ \textit{Kim-divides over $M$} if $p$ implies a formula $\phi(x,b)$ which \textit{Kim-divides over $M$}, that is, there exists an $\indi u$-Morley sequence $(b_i)_{i<\omega}$ in $\tp(b/M)$ such that $\set{\phi(x,b_i)\mid i<\omega}$ is inconsistent.
    \item We say that $a$ and $b$ are \textit{Kim-independent over $M$}, denoted $a\indi K_M b$ if $\tp(a/Mb)$ does not Kim-divide over $M$
\end{itemize}
\end{defn}

Recall that NSOP$_1$ theories can be characterized in terms of Kim-forking in the same way simple theories are characterized in terms of forking. For instance, a theory is NSOP$_1$ if and only if $\indi K$ is symmetric over models, see \cite{KR20} for more about NSOP$_1$ theories.
\begin{remark}
    It is easy to check that if $a\indi u_M b$ then $a\indi K _M b$. Note also that if $a\indi u _C b$ in $\M$, then in any reduct of $\M$ the type of $a$ over $Mb$ is finitely satisfiable in $M$ in the sense of the reduct.
\end{remark}

\begin{fact}[Strengthened independence theorem] \cite[Theorem 2.13]{kruckmanramsey18}.
    Assume that $T$ is NSOP$_1$ and $M\prec \M$, $c_1\indi K_M a$, $c_2\indi K_M b$, $a\indi K_M b$ and $c_1\equiv_M c_2$. Then there exists $c$ such that $c\equiv_{Ma}c_1$, $c\equiv_{Mb} c_2$ and $(a,b,c)$ is an independent triple over $M$ (by which we mean $a\indi K_M bc$, $b\indi K_M ca$ and $c\indi K_M ab$).
\end{fact}

We denote by $\indi f$ the usual forking independence relation. The theory of Kim-independence in general is defined over models because of its intrinsic definition which uses global coheir types which in general only exist over models. However, the whole theory of Kim-independence can be extended over arbitrary sets provided every set is an \textit{extension base for forking} i.e. $a\indi f _C C$ for all $a,C$, see \cite{dobrowolskiKimramsey2022}. In particular, we have the following:
\begin{fact}
    If $T$ is NSOP$_1$ and every small subset of $\M$ is an extension base for forking, then for all $a,b,C$ there exists $a'\equiv_C a$ such that $a'\indi K_C b$.
\end{fact}

Given three fields $C\seq A\cap B$, we denote by $A\indi \ACF_C\ B$ if the transcendence degree of $a$ over $C$ equals the transcendence degree of $a$ over $B$, for all tuples $a$ from $A$.
\begin{fact}\cite[Proposition 9.28]{KR20} \cite{Chatzidakis_1999}.
    Let $T$ be an NSOP$_1$ theory of fields, then $\indi K$ implies $\indi \ACF\quad$.
\end{fact}

The following is a technical lemma we will need later.  It is a variant of the `algebraically reasonable independence theorem' of \cite{kruckmanramsey18}.  

\begin{lem} \label{lem:NSOP1technical}
    Suppose $T$ is NSOP$_{1}$, $M \models T$, and we have $A \ind^{K}_{M} B$.  If we are given models $N_{0},N_{1}$ which contain $M$ and satisfy $N_{0} \equiv_{M} N_{1}$, $N_{0} \ind^{K}_{M} A$, and $N_{1} \ind^{K}_{M} B$, then there is some model $N$ such that $N \equiv_{MA} N_{0}$, $N \equiv_{MB} N_{1}$, $N \ind^{K}_{M} AB$ and, additionally, $A \ind^{K}_{N} B$.
\end{lem}

\begin{proof}
    Let $I = (B_{i})_{i < \omega}$ be an indiscernible sequence with $B_{0} = B$ which is both tree Morley over $M$ and over $N_{1}$, which exists by \cite[Proposition 6.4]{kaplan2021transitivity}.  By the chain condition, there is $I' \equiv_{MB_{0}} I$ such that $A \ind^{K}_{M} I'$ and that $I'$ is $MA$-indiscernible. Let $N_{1}'$ be a model with $N'_{1}I' \equiv_{MB_{0}} N_{1}I$.  Note that, by invariance, we have $N'_{1} \ind^{K}_{M} B$ and $I'$ is tree Morley over $M$. Thus, we have $N'_{1} \ind^{K}_{M} I'$ and, hence, by the independence theorem, there is some $N_{*} \models \mathrm{tp}(N_{0}/MA) \cup \mathrm{tp}(N'_{1}/MI')$ with $N_{*} \ind^{K}_{M} AI'$. Then $I'$ is both $N_{*}$-indiscernible and $MA$-indiscernible.  By Ramsey and compactness, extract an $N_{*}A$-indiscernible sequence $I''$ from $I'$.  Then $I''$ is also a tree Morley sequence over $N_{*}$ and $I'' \equiv_{MA} I'$.  Choose some model $N$ such that $NI' \equiv_{MA} N_{*}I''$.  Thus $I'$ is $NA$-indiscernible and is a tree Morley sequence over $N$ starting with $B$.  Thus, by witnessing, we have $A \ind^{K}_{N} B$. 
\end{proof}

\subsection{$T$ is NSOP$_4$}

We will work in the model companion $T^{+}$ of the two-sorted theory of $c$-nilpotent Lie algebras $(K,V,[\cdot,\cdot])$ assuming that $\mathrm{Th}(K) = T^{\dagger}$ is NSOP$_{1}$.  Recall that monster model is called $\mathbb{M}$.  We will try to deduce that $T^{+}$ is NSOP$_{4}$.  

\begin{defn}
 Suppose $n \geq 3$.  We say a theory $T$ has SOP$_{n}$ ($n$-strong order property), if there is some type $p(x,y)$ and an indiscernible sequence $(a_{i})_{i < \omega}$ satisfying the following:
 \begin{itemize}
     \item $(a_{i},a_{j}) \models p \iff i < j$. 
     \item $p(x_{0},x_{1}) \cup p(x_{1},x_{2}) \cup \ldots \cup p(x_{n-2},x_{n-1}) \cup p(x_{n-1},x_{0})$ is inconsistent. 
 \end{itemize}
 We say that $T$ is NSOP$_n$ if it does not have SOP$_n$.
\end{defn}

Now we will argue that $T^{+}$ is NSOP$_{4}$.  We begin by defining a notion of independence.  Given substructures $A$, $B$, and $C$ of $\mathbb{M}$ with $K(A) = K(B) = K(C) = K$, we will abuse notation and write $A \ind^{\otimes}_{C} B$ to mean that $V(A)$ and $V(B)$ are freely amalgamated over $V(C)$ in $V(\langle A B \rangle)$, where each vector space is viewed as an LLA over $K$.  Note that, by Lemma \ref{extension of scalars}, if $A$ is a substructure of $\mathbb{M}$ and $F$ is a field with $K(A) \subseteq F \subseteq K(\mathbb{M})$, then $\langle F, A \rangle$ may be identified with the structure $(F, F \otimes_{K(A)} V(A))$.  We will refer to the substructure of $\mathbb{M}$ generated in this way as the `extension of scalars' of $A$ to the field $F$.  

\begin{defn}
    Suppose $A$, $B$, and $C$ are algebraically closed subsets of $\mathbb{M}$ with $C \subseteq A \cap B$.  We define $A \ind_{C} B$ to hold if the following conditions are satisfied:
    \begin{enumerate}
        \item $K(A) \ind_{K(C)}^{\mathrm{K}} K(B)$.
        \item $K(\langle A,B \rangle) = K(A)K(B)$.
        \item $\tilde{A} \ind^{\otimes}_{\tilde{C}} \tilde{B}$ where $\tilde{A}$, $\tilde{B}$, and $\tilde{C}$ are the extension of scalars of $A$, $B$, and $C$ to the field $K(\langle A,B \rangle)$.
    \end{enumerate}    
    More generally, given arbitrary sets $A,B$ and $C$, we define $A \ind_{C} B$ to mean $\mathrm{acl}(AC) \ind_{\mathrm{acl}(C)} \mathrm{acl}(BC)$.  When $\mathrm{Th}(K(\mathbb{M}))$ is strictly NSOP$_{1}$, we will additionally restrict attention to the case that $K(C)$ is a model of $\mathrm{Th}(K(\mathbb{M}))$ in the language of the field sort. 
\end{defn}

%\begin{rem}
% Note that, by Lemma \ref{extension of scalars}, $\tilde{A} = \langle K(\langle A,B \rangle),A \rangle$ and likewise for $\tilde{B}$ and $\tilde{C}$.  
%\end{rem}

First, we observe that one cannot extend scalars and get something freely amalgamated unless the structures were freely amalgamated in the first place.

\begin{lem} \label{lem:freenessgoesdown}
    Suppose $M \models T_{0}$ and $A,B,C$ are substructures with $C \subseteq A,B \subseteq M$. Let $K = K(A)K(B)$ and $K'$ be any field with $K \subseteq K' \subseteq K(M)$ and let $A',B',C'$ be the associated extensions of scalars of $A$, $B$, and $C$ to LLAs over $K'$.  Then if $A' \ind^{\otimes}_{C'}B'$ as LLAs over $K'$, then $A \ind^{\otimes}_{C} B$ as LLAs over $K$ and $K(\langle A,B \rangle) = K$ in $M$.  
\end{lem}

\begin{proof}
    We may view $A' \otimes_{C'} B'$ as the sub-LLA over $K'$ of $M$ generated by $A'$ and $B'$.  In the two-sorted language, this amounts to saying that $A' \otimes_{C'} B'$ can be identified with the vector space sort of $\langle A',B' \rangle = \langle A,B,K' \rangle$ in $M'$.  Let $\tilde{A} = \langle A,K\rangle$, $\tilde{B} = \langle B,K \rangle$, and $\tilde{C} = \langle C,K\rangle$ be the extensions of scalars of $A$, $B$, and $C$ to $K$. We will argue that $\langle A,B \rangle$ can be identified with $\tilde{A} \otimes_{\tilde{C}} \tilde{B}$.  

    Let $L$ be any LLA over $K$ and consider LLA homomorphisms $f: \tilde{A} \to L$ and $g: \tilde{B} \to L$ over $K$ which agree on $\tilde{C}$.  Let $L'$ be the extension of scalars of $L$ to $L'$, an LLA over $K'$.  Then $f$ and $g$ induce $K'$-linear maps $f' : A' \to L'$ and $g' : B' \to L'$ which agree on $C'$.  Hence, there is a map $h' : \langle A',B' \rangle \to L'$ extending $f'$ and $g'$.  In particular, $h = h'|_{\langle A,B \rangle}$ extends $f$ and $g$.  Since the image of $h$ is, therefore, generated as an LLA over $K$ by $f(\tilde{A})$ and $g(\tilde{B})$, it follows that the image of $h$ lands in $L$.  Applying this to the $L = \tilde{A} \otimes_{\tilde{C}} \tilde{B}$ with $f= \mathrm{id}_{\tilde{A}}$ and $g = \mathrm{id}_{\tilde{B}}$, we see in particular that $K(\langle A,B \rangle) = K$. More generally, we have shown $\langle A,B \rangle$ satisfies the desired universal property so $\langle A,B \rangle \cong \tilde{A} \otimes_{\tilde{C}} \tilde{B}$.  
\end{proof}

\begin{lem} \label{existence}
    The relation $\ind$ is invariant, symmetric, and satisfies full existence (over models in the case that the field is strictly NSOP$_{1}$). Moreover, in the case that the field is algebraically closed, $\ind$ is stationary. 
\end{lem}

\begin{proof}
    Invariance and symmetry are clear, by the invariance and symmetry of $\ind^{K}$ in NSOP$_{1}$ theories.  To prove full existence, suppose we are given algebraically closed $A$, $B$, and $C$ with $C \subseteq A \cap B$. We will assume $K(C)$ is a model of $\mathrm{Th}(K(\mathbb{M}))$.  We want to find some $D \equiv_{C} A$ such that $D \ind_{C} B$.  First, we will use full existence in $\mathrm{Th}(K(\mathbb{M}))$ to find some $K' \equiv_{K(C)} K(A)$ with $K' \ind^{\mathrm{K}}_{K(C)} K(B)$.  Let $\sigma \in \mathrm{Aut}(K(\mathbb{M})/K(C))$ be an automorphism with $\sigma(K(A)) = K'$. Let $\overline{c}$ be a basis of $V(C)$.  Let $C' = \langle K(A), C \rangle$, so $V(C')$ is the extension of scalars of $V(C)$ to $K(A)$.  Let $\overline{a}$ be a $K(A)$-basis of $V(A)$ over $V(C')$ ($\overline{a}$ is possibly the empty tuple), so $\overline{a}\overline{c}$ is a basis of $V(A)$ as a $K(A)$-vector space.  Let $A'$ be the $2$-sorted LLA over $K'$ with basis $\overline{a}\overline{c}$, viewed naturally as an extension of $C'' = \langle K',C \rangle$, with structure induced from $A$ by the isomorphism $\sigma$, i.e. define $\tilde{\sigma} : A \to A'$ by 
    $$
    \sum \alpha_{i} a_{i} + \sum \beta_{j} c_{j} \mapsto \sum \sigma(\alpha_{i}) a_{i} + \sum \sigma(\beta_{j}) c_{j},
    $$
    and we define an LLA structure on $A'$ so that $\tilde{\sigma}$ is an isomorphism.  By quantifier elimination, we can embed $A'$ into $\mathbb{M}$ over $C''$.  Let $A''$ denote the image and let $\overline{a}''$ denote the image of $\overline{a}$. Note that $K(A'') = K'$.  

    Now let $\tilde{K} = K'K(B) = K(A'')K(B)$ and define $\tilde{A}'' = \langle \tilde{K},A'' \rangle$, $\tilde{B} = \langle \tilde{K},B \rangle$, and $\tilde{C} = \langle \tilde{K},C \rangle = \langle \tilde{K},C' \rangle$. By Fact \ref{fact:basicpropertiesoffreeindependence}, there is some $\tilde{D}$, isomorphic over $\tilde{C}$ to $\tilde{A}''$ as an LLA over $\tilde{K}$ such that $\tilde{D} \ind^{\otimes}_{\tilde{C}} \tilde{B}$.  By quantifier elimination, we may assume $\tilde{D}$ is embedded in $\mathbb{M}$ over $\tilde{B}$ and we have $\tilde{D} \equiv_{\tilde{C}} \tilde{A}''$.  Let $\overline{d}$ be the tuple corresponding to $\overline{a}''$ in $\tilde{D}$.  Then let $D = \langle K', \overline{d},\overline{c} \rangle$.  Note that, by construction, $D \equiv_{C} A$ and $\langle \tilde{K},D \rangle = \tilde{D}$.  Moreover, we have 
    $$
    \tilde{K} = K(D)K(B) \subseteq K(\langle D,B \rangle) \subseteq K(\langle \tilde{D},\tilde{B} \rangle) = \tilde{K},
    $$
    so $K(\langle D,B \rangle) = \tilde{K}$ and we have $D \ind_{C} B$ as desired. 

    Finally, we restrict to the case when the field is algebraically closed and prove stationarity. Suppose $A$, $B_{0}$, $B_{1}$ are algebraically closed sets containing an algebraically closed set $C$.  Suppose $B_{0} \equiv_{C} B_{1}$ and $A \ind_{C} B_{i}$ for $i = 0,1$. Let $g: B_{0} \to B_{1}$ be a $C$-isomorphism witnessing $B_{0} \equiv_{C} B_{1}$.  Define $\tilde{K}_{0} = K(\langle A,B_{0} \rangle)$ and $\tilde{K}_{1} = K(\langle A,B_{1} \rangle )$.  By stationarity of forking independence in ACF and the fact that $\tilde{K}_{0} = K(A)K(B_{0})$ and $\tilde{K}_{1} = K(A)K(B_{1})$, we know that $g|_{K(B_{0})}$ extends to a $K(A)$-isomorphism $f:\tilde{K}_{0} \to \tilde{K}_{1}$. Since $\langle \tilde{K}_{i} \rangle = \tilde{K}_{i}$ (i.e. $V(\langle \tilde{K}_{i} \rangle) = 0$) for $i = 0,1$, we know by quantifier elimination that $f$ extends to an automorphism $\sigma \in \mathrm{Aut}(\mathbb{M}/K(A))$.
    
    For $i = 0,1$, let $\tilde{A}_{i} = \langle \tilde{K}_{i},A \rangle$, $\tilde{B}_{i} = \langle \tilde{K}_{i},B_{i} \rangle$, and $\tilde{C}_{i} = \langle \tilde{K}_{i},C \rangle$.  Since $V(\tilde{A}_{i}) \cong V(A) \otimes_{K(A)} \tilde{K}_{i}$ for $i =0,1$, we have an isomorphism $\tilde{f}: \tilde{A}_{0} \to \tilde{A}_{1}$ induced by the isomorphism $v \otimes c \mapsto v \otimes f(c)$ for $v \in V(A)$ and $c \in \tilde{K}_{0}$.  Similarly, we have a map $\tilde{g}: \tilde{B}_{0} \to \tilde{B}_{1}$ induced by the isomorphism $V(B_{0}) \otimes_{K(B_{0})} \tilde{K}_{0} \to V(B_{1}) \otimes_{K(B_{1})} \tilde{K}_{1}$ given by $v \otimes c \mapsto g(v) \otimes f(c)$ for $v \in V(B_{0})$ and $c \in \tilde{K}_{0}$. By construction, $\tilde{f}|_{\tilde{C}_{0}} = \tilde{g}|_{\tilde{C}_{0}}$.  Let $L = \sigma^{-1}(\tilde{A}_{1} \otimes_{\tilde{C}_{1}} \tilde{B}_{1})$, which may be regarded as an LLA over $\tilde{K}_{0}$.
    
    The maps $\sigma^{-1} \circ \tilde{f}$ and $\sigma^{-1} \circ \tilde{g}$ may be regarded as homomorphisms of LLAs over $\tilde{K}_{0}$, which agree on $\tilde{C}_{0}$, and therefore, by the universal property of the free amalgam, there is a unique map $h: \tilde{A}_{0} \otimes_{\tilde{C}_{0}} \tilde{B}_{0} \to L$ extending $\sigma^{-1} \circ \tilde{f}$ and $\sigma^{-1} \circ \tilde{g}$ (note that it makes sense to say that this map extends $\sigma^{-1} \circ \tilde{f}$ and $\sigma^{-1} \circ \tilde{g}$ because we are identifying $\tilde{A}_{0} \otimes_{\tilde{C}_{0}} \tilde{B}_{0}$ with $\langle \tilde{A}_{0},\tilde{B}_{0} \rangle$).  By invariance, $L$ is the free amalgam of $\sigma^{-1}(\tilde{A}_{1})$ and $\sigma^{-1}(\tilde{B}_{1})$ over $\sigma^{-1}(\tilde{C}_{1})$.  The same argument, applied to the maps $\tilde{f}^{-1} \circ \sigma : \sigma^{-1}(\tilde{A}_{1}) \to (\tilde{A}_{0} \otimes_{\tilde{C}_{0}} \tilde{B}_{0})$ and $\tilde{g}^{-1} \circ \sigma : \sigma^{-1}(\tilde{B}_{1}) \to (\tilde{A}_{0} \otimes_{\tilde{C}_{0}} \tilde{B}_{0})$ implies that there is a map $L \to (\tilde{A}_{0} \otimes_{\tilde{C}_{0}} \tilde{B}_{0})$ extending $\tilde{f}^{-1} \circ \sigma$ and $\tilde{g}^{-1} \circ \sigma$, which must therefore be the inverse of $h$.  This shows $h$ is an isomorphism.

    It follows, then, that $\sigma \circ h : \tilde{A}_{0} \otimes_{\tilde{C}_{0}} \tilde{B}_{0} \to \tilde{A}_{1} \otimes_{\tilde{C}_{1}} \tilde{B}_{1}$ is an isomorphism. Since $\sigma \circ h$ extends $\tilde{f}$, which was defined to fix both $V(A)$ and $K(A)$, we know that $\sigma\circ h$ fixes $A$. Additionally, as $\sigma \circ h$ extends $\tilde{g}$, it takes $B_{0}$ to $B_{1}$.  By quantifier elimination, this shows $B_{0} \equiv_{A} B_{1}$, which proves stationarity. 
\end{proof}

\begin{rem}
    The increasingly standard approach for showing that a structure has NSOP$_{4}$ is to deduce NSOP$_{4}$ from the existence of a stationary independence relation. However, this approach cannot possibly work to show that the theory $T^{+}$ is NSOP$_{4}$ when $T^{\dagger}$ is, for example, the theory of pseudo-finite fields since these fields themselves do not have any stationary independence relations.  To see this, note that the usual axioms of a stationary independence relation entail that if $a_{1}, \ldots, a_{n}$ is an independent sequence of tuples with the same type over $E$, then $a_{1}\ldots a_{n} \equiv_{E} a_{\sigma(1)}\ldots a_{\sigma(n)}$ for all $\sigma \in \mathfrak S _{n}$ (and, indeed, the $n=2$ case is the key property used in proofs of NSOP$_{4}$). But Beyarslan and Hrushovski \cite{beyarslan2012algebraic} show that in almost every pseudo-finite field $F$, there is a definable $p$\emph{-tournament} for some prime $p$, which is a definable relation $R(x_{1}, \ldots, x_{p})$ such that, for any $p$ distinct elements $a_{1}, \ldots, a_{p} \in F$, there is a unique $\sigma \in \mathfrak S_{p}$ such that $F \models R(a_{\sigma(1)}, \ldots, a_{\sigma(p)})$.  This is clearly incompatible with the existence of a stationary independence relation. We note, moreover, that there \emph{is} a stationary independence relation on $\omega$-free PAC fields, defined by Chatzidakis in \cite{Chatzidakis2002properties}, which allows one to define a stationary independence relation for $T^{+}$ when $T^{\dagger}$ is the theory of an $\omega$-free PAC field.  This has the amusing consequence that it is \emph{easier} to prove NSOP$_{4}$ for $c$-nilpotent LLAs over $\omega$-free PAC fields than it is over pseudo-finite fields. 
\end{rem}

The proof of full existence gives a bit more:

\begin{lem} \label{lem:strongextension}
    Suppose $A$, $B$, and $C$ are algebraically closed sets with $C \subseteq A,B$.  If $K(A) \ind^{K}_{K(C)} K(B)$, then there is some $A' \equiv_{CK(A)} A$ such that $A' \ind_{C} B$.  
\end{lem}

\begin{proof}
    Our assumption gives us that, in the preceding proof, we may take $K' = K(A)$ and $\sigma$ to be the identity. Then the $D$ constructed there will have $D \equiv_{CK(A)} A$.  
\end{proof}

We will also define a notion of weak independence for $T^{+}$:

\begin{defn}
 Suppose $A$ and $B$ are algebraically closed subsets of $\mathbb{M}$ and $M$ is a model with $M \subseteq A \cap B$.  We say that $A$ and $B$ are \emph{weakly independent} if $K(A) \ind^{K}_{K(M)} K(B)$ and $V(A)$ and $V(B)$ are linearly independent over $V(M)$ (in $V(\mathbb{M})$).  We denote this by $A \ind^{w}_{M} B$.  More generally, we may write $a \ind^{w}_{M} b$ to mean $\mathrm{acl}(aM) \ind^{w}_{M} \mathrm{acl}(bM)$. 
\end{defn}

Notice that, for a model $M$, we have $a \ind_{M} b$ implies $a \ind_{M}^{w} b$, which implies $a \ind^{a}_{M} b$.  

\begin{lemma}\label{lm:WT} The relation $\ind$ satisfies ``weak transitivity" over models:
    \[a\ind_{Md} b \text{ and } a\ind^{w}_M d \text{ and } d\ind^{w}_M b  \implies a\ind_M b\]
\end{lemma}

\begin{proof}
Let $M$ be a model of $T^+$ and assume $a$, $b$, $c$, and $d$ are finite tuples satifying the hypotheses of the statement of ``weak transitivity".  Let $C = \mathrm{acl}(Md)$, $A = \mathrm{acl}(aC)$, and $B = \mathrm{acl}(bC)$.  Additionally, set $A_{0} = \mathrm{acl}(aM)$ and $B_{0} = \mathrm{acl}(bM)$.  

        Our assumptions imply $K(A) \ind^{K}_{K(C)} K(B)$ and $K(A_{0}) \ind^{K}_{K(M)} K(C)$.  Therefore, by transitivity and monotonicity, we get the desired $K(A_{0}) \ind^{K}_{K(M)} K(B_{0})$.  

        The proofs of \cite[Lemma 5.11]{d2023model} and \cite[Corollary 5.12]{d2023model} readily adapt to the two-sorted case, though we give details on one point.   The claim of \cite[Lemma 5.11]{d2023model} is that if, given LLAs $E,F,G$ over a fixed field with $G \subseteq E,F$, then, assuming $E$ and $F$ are either heir or coheir independent over $G$, then it follows that if $e_{1}, \ldots, e_{n} \in E$ and $\langle F, e_{<i}\rangle$ is an ideal of $\langle F, e_{\leq i} \rangle$ for all $i \leq n$, then $\langle G, e_{<i}\rangle$ is an ideal of $\langle G,e_{\leq i} \rangle$ for all $i \leq n$. We claim this follows the weaker assumption that $E$ and $F$ are \emph{algebraically} independent over $G$.  To see this, assume $E$ and $F$ are algebraically independent over $G$ and suppose we are given $e_{1}, \ldots, e_{n} \in E$ such that $\langle F, e_{<i}\rangle$ is an ideal of $\langle F, e_{\leq i} \rangle$ for all $i \leq n$. Assume for induction that we have shown for some $m < n$ that $\langle G, e_{<i}\rangle$ is an ideal of $\langle G,e_{\leq i} \rangle$ for all $i \leq m$.  For the induction step, we must show $[e_{m+1},v] \in \langle G, e_{\leq m}\rangle$ for $v \in \langle G, e_{\leq m} \rangle$.  Write $v = g + \sum_{i = 1}^{m} \lambda_{i} e_{i}$. By our hypothesis, $[e_{m+1}, v] \in \langle F, e_{\leq m} \rangle$ so we may write
        $$
        [e_{m+1},v] = f + \sum_{i = 1}^{m} \mu_{i}e_{i},
        $$
        for some $f \in F$. But then $f = [e_{m+1},v] - \sum_{i=1}^{m} \mu_{i}e_{i}\in E \cap F = G$, by our assumption of algebraic independence, so we have $[e_{m+1},v] \in \langle G, e_{\leq m} \rangle$ as desired. 
        
        The argument of \cite[Theorem 5.13]{d2023model}, then, gives us that $\tilde{A}'_{0} \ind_{\tilde{M}'}^{\otimes} \tilde{B}_{0}'$ where $\tilde{A}'_{0} = \langle K(\langle A,B \rangle),A \rangle$, $\tilde{B}_{0}' = \langle K(\langle A,B \rangle), B \rangle$, and $\tilde{M}' = \langle K(\langle A,B \rangle), M \rangle$.  We claim that this implies $A_{0} \ind_{M} B_{0}$.  We are left to show that if $\tilde{A}_{0} = \langle K(\langle A_{0},B_{0} \rangle),A\rangle$, $\tilde{B}_{0} = \langle K(\langle A_{0},B_{0}),B_{0}\rangle$, and $\tilde{M} = \langle K(\langle A_{0},B_{0} \rangle),M\rangle$, then $\tilde{A}_{0}  \ind^{\otimes}_{\tilde{M}} \tilde{B}_{0}$, and this follows immediately by Lemma \ref{lem:freenessgoesdown}. %Suppose $L$ is an LLA over $K(\langle A_{0},B_{0} \rangle)$ and $\alpha: \tilde{A}_{0} \to L$ and $\beta: \tilde{B}_{0} \to L$ are homomorphisms of LLAs over $K(\langle A_{0},B_{0} \rangle)$ which agree on $\tilde{M}$.  We obtain maps $\alpha' : \tilde{A}_{0}' \to L \otimes_{K(\langle A_{0},B_{0} \rangle)} K(\langle A,B \rangle)$ and $\beta': \tilde{B}_{0}' \to L \otimes_{K(\langle A_{0},B_{0} \rangle)} K(\langle A,B \rangle)$. By freeness, we get a map $\gamma: \langle K(\langle A,B \rangle), A_{0},B_{0} \rangle \to L \otimes_{K(\langle A_{0},B_{0} \rangle)} K(\langle A,B \rangle)$ extending $\alpha'$ and $\beta'$. It follows that $\gamma |_{A_{0}} = \alpha$ and $\gamma|_{\beta_{0}} = \beta$ and therefore $\gamma|_{\langle A_{0},B_{0} \rangle}$ is a map to $L$ extending $\alpha$ and $\beta$. As $L$ is arbitrary, this implies that $\langle A_{0},B_{0} \rangle = \langle \tilde{A}_{0},\tilde{B}_{0} \rangle$ is the free amalgam of $\tilde{A}_{0}$ and $\tilde{B}_{0}$ over $\tilde{C}_{0}$, as desired. 
\end{proof}

\begin{prop} \label{prop:weakIT}
    Suppose $M \models T^+$, $A = \mathrm{acl}(AM)$, $B = \mathrm{acl}(BM)$, and $A \ind_{M} B$.  If $C_{i} = \mathrm{acl}(C_{i}M)$ and $K(C_{i}) \models \mathrm{Th}(K(\mathbb{M}))$ for $i = 0,1$, $C_{0} \equiv_{M} C_{1}$, $C_{0} \ind^{w}_{M} A$, $C_{1} \ind^{w}_{M} B$, then there is $C_{*}$ such that $C_{*} \equiv_{A} C_{0}$ and $C_{*} \equiv_{B} C_{1}$. 
\end{prop}

\begin{proof}
    First, we use NSOP$_{1}$ in the field sort to find some $K_{*}$ such that 
    $$
    K_{*} \models \mathrm{tp}_{L^{\dagger}}(K(C_{0})/K(A)) \cup \mathrm{tp}_{L^{\dagger}}(K(C_{1})/K(B)).
    $$
    By Lemma \ref{lem:NSOP1technical}, we may choose $K_{*}$ so that $K(A) \ind^{K}_{K_{*}} K(B)$. 
 By quantifier elimination relative to the field sort (see Corollary \ref{cor:typedecomposition} or $(\star)$ above the latter), choose some $C_{*} \equiv_{M} C_{0} \equiv_{M} C_{1}$ so that $K(C_{*}) = K_{*}$. Notice that we have $C_{*} \equiv_{MK(A)} C_{0}$ and $C_{*} \equiv_{MK(B)} C_{1}$, by Lemma \ref{lem:better base}. 
    
    Choose some $\sigma_{0} \in \mathrm{Aut}(\mathbb{M}/MK(A))$ with $\sigma_{0}(C_{0}) = C_{*}$ and set $K_{0} = \sigma_{0}(K(\langle A,C_{0} \rangle))$.  Likewise, choose some $\sigma_{1} \in \mathrm{Aut}(\mathbb{M}/MK(B))$ with $\sigma_{1}(C_{1}) = C_{*}$ and $K_{1} = \sigma_{1}(K(\langle B,C_{1} \rangle))$.  Recall that we have already arranged that $K(A) \ind^{K}_{K_{*}} K(B)$.  By extension and symmetry, we may choose $K_{0}' \equiv_{K(A)K_{*}} K_{0}$ and $K'_{1} \equiv_{K(B)K_{*}} K_{1}$ such that $K'_{0} \ind^{K}_{K_{*}} K'_{1}$.  Necessarily, we have also, then, $K_{0}' \equiv_{K(A)C_{*}} K_{0}$ and $K'_{1} \equiv_{K(B)C_{*}} K_{1}$ since the field sort is stably embedded. So without loss of generality, $K_{0} = K_{0}'$ and $K_{1} = K_{1}'$.  

    Notice that, since we have $K(\langle A,C_{0} \rangle)C_{0} \equiv_{MK(A)} K_{0}C_{*}$ and $K(\langle B,C_{1} \rangle)C_{1} \equiv_{MK(B)} K_{1}C_{*}$, we can find $A'$ and $B'$ such that 
    $$
    K(\langle A,C_{0} \rangle) A C_{0} \equiv_{MK(A)} K_{0}A'C_{*}
    $$
    and
    $$
    K(\langle B,C_{1} \rangle) BC_{1} \equiv_{MK(B)} K_{1}B'C_{*}. 
    $$
    By construction, $K(\langle A',C_{*} \rangle) = K_{0}$ and $K(\langle B',C_{*} \rangle) = K_{1}$.  Thus, applying Lemma \ref{lem:strongextension}, we may assume $A' \ind_{C*} B'$. 

   Note that we have $\tilde{A}' \ind^{\mathrm{w}}_{\tilde{M}} \tilde{C}_{*}$ and $\tilde{B}' \ind^{\mathrm{w}}_{\tilde{M}} \tilde{C}_{*}$, where the tilde indicates the extension of scalars up to $\tilde{K} = K_{0}K_{1}$. This follows from the fact that linear independence is preserved under extension of scalars. Hence, by Lemma \ref{lm:WT}, we have $\tilde{A}' \ind^{\otimes}_{\tilde{M}} \tilde{B}'$ again as LLAs over $\tilde{K}$.  From this it follows, by Lemma \ref{lem:freenessgoesdown}, that $\langle A',B' \rangle$ is the free amalgam of $A''$ and $B''$ over $M''$, where $A''$, $B''$, and $M''$ are LLAs over $K(\langle A',B' \rangle) = K(A')K(B') = K(A)K(B)$ obtained from $A'$, $B'$, and $M$ by extension of scalars.  
   
   But, as $A \equiv_{M} A'$ and $K(A) = K(A')$, we have, by Lemma \ref{lem:better base} that $A \equiv_{MK(A)} A'$ and therefore there is an isomorphism $f$ over $MK(A)$ from $A$ to $A'$.  By extending scalars to $K(A)K(B)$, $f$ lifts uniquely to an $K(A)K(B)$-isomorphism $f'$ from $A''$ to the extension of scalars of $A''$ to an LLA over $K(A)K(B)$.  We may regard $f''$ as an embedding of LLAs over $\langle M,K(A),K(B)\rangle$ from $A''$ to $\langle A',B' \rangle$, with $f(A) = f(A')$. Arguing similarly, we can find some $K(A)K(B)$-embedding $g: B'' \to \langle A',B'\rangle$ of LLAs over $\langle M,K(A),K(B)\rangle$ with $g(B) = B'$. By the universal property for the free amalgam, there is a unique $K(A)K(B)$-embedding of LLAs $h : \langle A,B \rangle \to \langle A',B' \rangle$ over $M$ extending $f$ and $g$. By a symmetric argument applied to $f^{-1}$ and $g^{-1}$, we see that $h$ must be an isomorphism.  In other words, we have shown there is some $\sigma \in \mathrm{Aut}(\mathbb{M}/MK(A)K(B))$ with $\sigma(A'B') = AB$.  Let $C_{**} = \sigma(C_{*})$.     

   Note that we have $A'C_{*} \equiv_{M} AC_{**}$ and $B'C_{*} \equiv_{M} BC_{*}$, hence 
   $$
   C_{**} \models \mathrm{tp}(C_{0}/A) \cup \mathrm{tp}(C_{1}/B)
   $$
   as desired. 
\end{proof}

\begin{thm}
    The theory $T^{+}$ is NSOP$_{4}$.
\end{thm}

\begin{proof}
    By Fact \ref{fact:indiscernibleisheirovermodel}, it suffices to show that if $M \models T^{+}$ and $(A_{i})_{i < \omega}$ is a coheir sequence over $M$ consisting of models of $T$ that contain $M$ then, setting $p(X,Y) = \mathrm{tp}(A_{0},A_{1}/M)$, we have 
    $$
    p(X_{0},X_{1}) \cup p(X_{1},X_{2}) \cup p(X_{2},X_{3}) \cup p(X_{3},X_{0})
    $$
    is consistent.  

    Choose $A_{2}' \equiv_{A_{1}} A_{2}$ with $A_{2}' \ind_{A_{1}} A_{0}$.  By weak transitivity, we have $A_{2}' \ind_{M} A_{0}$.  

    Pick $A'_{1}$ such that 
    $$
    A'_{1}A_{2}' \equiv_{M} A_{2}A_{1} \left( \equiv_{M} A'_{2}A_{1} \right).
    $$
    Note that we have $\models p(A_{2}',A'_{1})$ and $A'_{1} \ind^{w}_{M} A'_{2}$.  

    Likewise, choose $A''_{1}$ such that 
    $$
    A''_{1}A_{0} \equiv_{M} A_{0}A_{1}.
    $$
    Then again we have $\models p(A''_{1},A_{0})$ and $A''_{1} \ind^{w}_{M} A_{0}$.  Since $A'_{1} \equiv_{M} A''_{1}$, we may apply Proposition \ref{prop:weakIT} to obtain some $A_{*}$ such that 
    $$
    A_{*} \models \mathrm{tp}(A'_{1}/A'_{2}) \cup \mathrm{tp}(A''_{1}/A_{0}).
    $$
    Therefore we have 
    $$
    \models p(A_{0},A_{1}) \cup p(A_{1},A'_{2}) \cup p(A'_{2},A_{*}) \cup p(A_{*},A_{0}),
    $$
    completing the proof.
\end{proof}

\begin{quest}
    What is Conant-independence in $T$?  We conjecture that it is field independence plus vector space independence for the algebraic closures\textemdash that is, we conjecture Conant-independence coincides with $\ind^{w}$ defined above, over models.  
\end{quest}

As a concluding remark, we extract the following criterion for NSOP$_{4}$, essentially what was used in the above proof, since this may be of independent interest: 

\begin{theorem}\label{thm:criterionNSOP4}
Let $T$ be any theory and $\ind$ be an invariant relation satisfying the following properties:
\begin{itemize}
    \item (Symmetry) $a\ind_Cb$ if and only if $b\ind_C a$;
    \item (Full Existence) for all $a,b,C$ there exists $a'\equiv_C a$ with $a'\ind_C b$;
    \item (Weak transitivity over models) if $a\ind_{Md} b$, $ a\indi h _M d$ and $b\indi u _M d$ then $a\ind_M b$;
    \item (Weak independence theorem over models) if $c_1\equiv_M c_2$ and $c_1\indi h_M a$, $c_2\indi u_M b$ and $a\ind_M b$
      then there exists $c$ with $c\equiv_{Ma} c_1$, $c \equiv_{Mb} c_2$.
\end{itemize} 
Then $T$ is NSOP$_4$.
\end{theorem}

\begin{proof}
     Let $(a_i)_{i<\omega}$ be an indiscernible sequence and for $p(x,y) = \tp(a_0,a_1)$, we need to prove that 
    \[p(x_0,x_1)\cup p(x_1,x_2)\cup p(x_2,x_3)\cup p(x_3,x_0)\]
    is consistent. By Fact \ref{fact:indiscernibleisheirovermodel}, there is a model $M$ such that $a_i\indi h_M a_{<i}$ for all $i$. By full existence, there exists $a_2^*\equiv_{Ma_1} a_2$ such that $a_2^*\ind_{Ma_1} a_0$. From $a_2\indi h _M a_1$ and $a_1\indi h_M a_0$ we get $a_2^*\indi h _M a_1$ by invariance and $a_0\indi u_M a_1$ by definition. Using weak transitivity, we conclude $a_2^*\ind_M a_0$. As $a_0\equiv_M a_1$ there exists $c_1$ such that $c_1a_0\equiv_M a_0a_1$ and by invariance, $c_1\indi u_M a_0$. Similarly, as $a_2^*\equiv_M a_1$ there exists $c_2$ such that $a_1a_2^*\equiv_M a_2^*c_2$ and by invariance $c_2\indi h_M a_2^*$. We have $c_1\equiv_M c_2$, $c_1\indi u _M a_0$, $c_2\indi h _M a_2^*$ hence by the weak independence theorem over models, we conclude that there exists $a_3^*$ such that $a_3^*a_0\equiv_M c_1a_0\equiv_M a_0a_1$ and $a_3^*a_2^*\equiv_M c_2a_2^*\equiv_M a_2^*a_1\equiv_M a_2a_1$. As $a_0a_1\equiv a_1a_2$, we conclude that 
    \[a_0a_1\equiv a_1a_2^*\equiv a_2^* a_3^*\equiv a_3^*a_0\]
    hence the type above is consistent.
\end{proof}

\begin{quest}
It would be interesting to prove that $T^{+}$ is NSOP$_{4}$ in cases where the theory of the field is \emph{not} assumed to be NSOP$_{1}$. Some precise variants of this question: 
    \begin{enumerate}
        \item Suppose $T^{\dagger}$ is NSOP$_{4}$.  Does it follow that $T^{+}$ is NSOP$_{4}$?
        \item Suppose $T^{\dagger}$ has symmetric Conant-independence. Does it follow that $T^{+}$ is NSOP$_{4}$?
        \item Suppose $T^{\dagger}$ is a theory of curve-excluding fields.  Does it follow that $T^{+}$ is NSOP$_{4}$?
    \end{enumerate}
\end{quest}

\subsection*{Acknowledgements}

We are grateful to the anonymous referee for numerous helpful comments and suggestions. Much of this work was completed at the Institut Henri Poincar\'e when the authors were hosted as part of a `Research in Paris' program. The authors would also like to thank the Institut Henri Poincar\'e for their hospitality.

\subsection*{Funding}

D'Elb\'ee was partially supported by the UKRI Horizon Europe Guarantee Scheme, grant no EP/Y027833/1, and by the Ramon y Cajal grant RYC2023-042677-I funded by MICIU/AEI/10.13039/501100011033 and by ESF+. M\"uller was supported by the Faculty Research Support Grant of AUC. Ramsey was supported by NSF grant DMS-2246992.

\bibliographystyle{plain}
\bibliography{biblio.bib}{}

\end{document}